\documentclass[12pt,reqno,sumlimits]{amsart}
\usepackage{amssymb,amscd, epsfig, mathrsfs, xypic, amsmath,amsthm, dsfont, txfonts, arydshln}
\xyoption{all} 
\theoremstyle{plain}
\newtheorem{theorem}{Theorem}[section]
\newtheorem{lemma}[theorem]{Lemma}
\newtheorem{cor}[theorem]{Corollary}
\newtheorem{prop}[theorem]{Proposition}

\newtheorem{thmx}{Theorem}

\theoremstyle{definition}
\newtheorem{defn}[theorem]{Definition}
\newtheorem{example}[theorem]{Example}

\theoremstyle{remark}
\newtheorem{remark}[theorem]{Remark}

\numberwithin{equation}{section}
\newcommand{\bbf}{{\mathbf f }}
\newcommand{\be}{{\mathbf e }}

\newcommand{\ft}{{\mathfrak t }}
\newcommand{\fs}{{\mathfrak s}}

\newcommand{\cI}{{\mathcal I}}
\newcommand{\calI}{{\mathcal I}}

\newcommand{\calH}{{\mathcal H}}
\newcommand{\calNH}{{\mathcal{NH}}}

\newcommand{\calE}{{\mathcal E}}

\newcommand{\calX}{{\mathcal X}}

\newcommand{\calR}{{\mathcal R}}
\newcommand{\cS}{{\mathcal S}}

\newcommand{\calS}{{\mathcal S}}
\newcommand{\calL}{{\mathcal L}}

\newcommand{\bv}{{\mathbf v}}
\newcommand{\bw}{{\mathbf w}}
\newcommand{\bu}{{\mathbf u}}

\newcommand{\into}{\hookrightarrow}

\newcommand{\sfS }{{\mathsf S}}
\newcommand{\sfT}{{\mathsf T}}
\newcommand{\sfG}{{\mathsf G}}
\newcommand{\sfR}{{\mathsf R}}

\newcommand{\fraks}{{\mathfrak s}}

\newcommand{\frakr}{{\mathfrak r}}

\newcommand{\R}{{\mathbb R}}

\newcommand{\Q}{{\mathbb Q}}
\newcommand{\C}{{\mathbb C}}

\newcommand{\Z}{{\mathbb Z}}

\newcommand{\bbP}{{\mathbb P}}

\newcommand{\lan}{{\langle}}
\newcommand{\ran}{{\rangle}}

\newcommand{\inc}{\hookrightarrow}
\newcommand{\surj}{\twoheadrightarrow}
\newcommand{\Mat}{\operatorname{Mat}}

\newcommand{\SR}{\operatorname{SR}}
\newcommand{\Crit}{\operatorname{Crit}}

\newcommand{\Sym}{\operatorname{Sym}}

\newcommand{\Hom}{\operatorname{Hom}}

\newcommand{\Lie}{\operatorname{Lie}}

\newcommand{\im}{{\operatorname{Im\ }}}

\newcommand{\Ann}{{\operatorname{Ann\ }}}

\newcommand{\U}{{\mbox{U}}}

\newsavebox{\savepar}

\newcounter{labelflag} 
\newcommand{\labelon}{\setcounter{labelflag}{1}}
\newcommand{\Label}[1]{\ifnum\thelabelflag=1\ifmmode
\makebox[0in][l]{\qquad\fbox{\rm#1}} \else
\marginpar{\vspace{0.7\baselineskip} \hspace{-1.1\textwidth}
\fbox{\rm#1}} \fi \fi \label{#1} } \labelon
\begin{document}

\title[Equivariant cohomology of orbifolds]
{Equivariant cohomology for Hamiltonian\\ torus actions on symplectic orbifolds}

\author{Tara Holm}
\address{Department of Mathematics\\
Cornell University\\
Ithaca, NY 14853, USA}
\email{tsh@math.cornell.edu}

\author{Tomoo Matsumura}
\address{Department of Mathematical Sciences\\
ASARC, KAIST\\
Daejeon, South Korea}
\email{tomoomatsumura@kaist.ac.kr}
\thanks{The first author gratefully acknowledges the support of the NSF through Grant DMS-0835507.  The second author is supported by the National Research
Foundation of Korea (NRF) through Grants No. 2012-0000795 and 2011-0001181 funded by the Korea government (MEST)}

\maketitle

\begin{abstract}
We study Hamiltonian $\sfR$-actions on symplectic orbifolds $[M/\sfS ]$, where $\sfR$ and $\sfS $ are tori. We prove an injectivity theorem and generalize Tolman and Weitsman's proof of the GKM theorem \cite{TW} in this setting. The main example is the symplectic reduction $X/\!/\sfS $ of a Hamiltonian $\sfT$-manifold $X$ by a subtorus $\sfS \subset \sfT$.  This includes the class of {\bf symplectic toric orbifolds}. We define the equivariant Chen-Ruan cohomology ring and use the above results to establish a combinatorial method of computing this equivariant Chen-Ruan cohomology in terms of {\bf{orbifold fixed point data}}.
\end{abstract}
%%   We insist on abstract in your paper!

%%%%%%%%%%%%%%%%%%%%%%%%%%%%%%%%%%
%%%%%%%%%%%%%%%%%%%%%%%%%%%%%%%%%%
%%%%%%%%%%%%%%%%%%%%%%%%%%%%%%%%%%
\section{Introduction}
%%%%%%%%%%%%%%%%%%%%%%%%%%%%%%%%%%
%%%%%%%%%%%%%%%%%%%%%%%%%%%%%%%%%%
%%%%%%%%%%%%%%%%%%%%%%%%%%%%%%%%%%
There has been a flurry of recent work computing a variety of algebraic invariants for orbifolds.
In the present paper, we consider {\bf equivariant} invariants of an orbifold equipped with a
group action.  Our orbifolds arise as global quotients $[M/\sfS]$ of a manifold by a torus 
acting with finite stabilizers, and our actions of a torus $\sfR$ on $[M/\sfS]$ arise as extensions of the action of 
$\sfS$ on $M$. We will discuss ordinary and stringy equivariant invariants, and relate our results to the current 
literature.  We include several explicit examples coming from the symplectic reduction 
construction in symplectic geometry.

Let $\sfT\cong S^1\times \cdots \times S^1$ be a compact torus and $\sfS  \subset \sfT$ a subtorus, with the quotient torus $\sfR:=\sfT/\sfS $. Let $M$ be a $\sfT$-manifold such that $[M/\sfS ]$ is a compact $\sfR$-Hamiltonian orbifold in the sense of \cite{LT}.  We note that topological invariants of the orbifold $[M/\sfS ]$ should be $\sfS$-equivariant invariants of $M$.  Hence the $\sfR$-equivariant cohomology of the orbifold $[M/\sfS ]$ is defined to be
\[
H_{\sfR}^*([M/\sfS ],\Z):=H^*_{\sfT}(M,\Z).
\]
Note that on the level of topological spaces, $H^*_{\sfT}(M,\Q)=H^*_{\sfR}(M/\sfS ,\Q)$ but over $\Z$ they are not equal in general. See \cite{H} for this kind of comparison.

The main application comes from symplectic reduction. Let $(X, \omega)$ be a connected, Hamiltonian $\sfT$-manifold with moment map $\mu_{\sfT}: X\to \ft^*$, where $\ft$ denotes the Lie algebra of $\sfT$.  We assume that $\mu_{\sfT}$ has a component that is proper and bounded below. For a subtorus $\sfS \subset \sfT$, containing the proper component, we have the natural inclusion $\fs\into\ft$ of Lie algebras, and we let $\pi_S: \ft^*\to\fs^*$ denote the dual projection.  Then $X$ is also a Hamiltonian $\sfS $-manifold with the induced moment map $\mu_{\sfS}=\pi_S\circ\mu_{\sfT}: X\to\fraks^*$.   For a regular value $a \in \fraks^*$, we let 
$M:=\mu_{\sfS}^{-1}(a)$ be the level set of the $\sfS$ moment map.  The {\bf symplectic reduction}
$X/\!/\sfS :=[M/\sfS ]$ of $X$ by the action of $\sfS $ at $a$ is a compact {\bf symplectic orbifold} which is Hamiltonian with respect to the residual torus action of $\sfR:=\sfT/\sfS$, in the sense of \cite{LT}.  

Classically,  if $Y$ is a compact Hamiltonian $\sfT$-manifold satisfying the GKM conditions, the GKM theorem \cite{GKM} computes the $\sfT$-equivariant cohomology of $Y$ in terms of the fixed point data. Our first main result is to generalize the GKM theorem to compute $H_{\sfR}^*([M/\sfS ],\Z)$ by adopting the proof in \cite{TW} to our setting as follows.  Please note that the complete technical hypotheses of the theorems appear in the main body of this paper as noted.

%%%%%%%%%%%%%%%%%%%%%%%%%%%%%%%%%%
\begin{thmx}[Theorems \ref{thm:injectivity} and \ref{orbifold GKM} below]
Let $[M/\sfS ]^{\sfR}$ be the suborbifold consisting the $0$-dimensional $\sfR$-orbits in $[M/\sfS ]$, and let $[M_1/\sfS ]$ the {\bf orbifold 1-skeleton}, the suborbifold consisting of the $0$- and $1$-dimensional $\sfR$-orbits in $[M/\sfS ]$. We have the following diagram of $\sfR$-equivariant natural inclusions
\[
\begin{array}{c}
\xymatrix{
[M/\sfS ]  & [M_1/\sfS ]\ar[l] \\
& [M/\sfS ]^{\sfR} \ar[u]_j \ar[lu]^i
}
\end{array}.
\]
When we take $R$-equivariant cohomology, the image of the injection
\[
i^*: H^*_{\sfR}([M/\sfS ],\Q) \to H^*_{\sfR}([M/\sfS ]^{\sfR},\Q)
\] 
is the same as the image of 
\[
j^*: H^*_{\sfR}([M_1/\sfS ],\Q)\to H^*_{\sfR}([M/\sfS ]^{\sfR},\Q).
\]
The map $i^*$ is injective in cohomology with integer coefficients, and
the images of $i^*$ and $j^*$ coincide in cohomology with integer coefficients,
when the stabilizer subgroups are connected and the isotropy weights are primitive.
\end{thmx}
%%%%%%%%%%%%%%%%%%%%%%%%%%%%%%%%%%

Applying this theorem to compact symplectic toric orbifolds, we obtain the following.
%%%%%%%%%%%%%%%%%%%%%%%%%%%%%%%%%%

\begin{thmx}[Theorem \ref{stanley reisner for orbifold} below]
Let $\sfS $ be an $(m-n)$-dimensional subtorus of the $m$-dimensional torus $\sfT$ which acts on $\C^m$ canonically coordinate-wise. Let $\Delta$ be the moment polytope of the compact toric orbifold obtained as the symplectic reduction $\C^m/\!/ \sfS $ at a regular value. Then
\[
H^*_{\sfR}\left(\C^m/\!/ \sfS ,\Z\right) \cong \SR(\Delta)
\]
where $\SR(\Delta)$ is the Stanley-Reisner ring of the polytope $\Delta$. Note that this Stanley-Reisner description depends only on the combinatorial type of $\Delta$.
\end{thmx}
%%%%%%%%%%%%%%%%%%%%%%%%%%%%%%%%%%
In the second part of this paper, we introduce the $\sfR$-equivariant version of the Chen-Ruan orbifold cohomology 
ring, denoted   $H_{orb,\sfR}([M/\sfS ])$, for a Hamiltonian $\sfR$-orbifold $[M/\sfS ]$. We use our orbifold GKM theorem 
to compute $H_{orb,\sfR}([M/\sfS ])$ in terms of fixed point data.  Our definitions coincide with those found in the literature;  
we discuss this further at the end of Section~\ref{sec:eqCR}.

As a vector space, the {\bf $\sfR$-equivariant Chen-Ruan orbifold cohomology} is defined to be 
\[
H_{orb,\sfR}([M/\sfS ]):= \bigoplus_{g \in \sfS } H_{\sfR}^*([M^g/\sfS ]).
\]
As a graded ring, we must add a shifted grading and define a twisted product given by the usual pull-cup-push formula (see, for example,  \cite{FG,GHK,JKK})
\[
\eta \odot \xi := e_* \left(e_1^*\eta \cup e_2^*\xi \cup c_{M}(g,h)\right) \ \ \ \ \mbox{ for } (\eta,\xi) \in H^*_{\sfR}([M^g/\sfS ])\times H^*_{\sfR}([M^h/\sfS ]),
\]
where $e_1,e_2,e$ are the obvious inclusions of $M^{g,h}:=M^g\cap M^h$ to $M^g$, $M^h$, and $M^{gh}$ respectively. Here we adapt the formula from \cite{BCS, EJK} to define  the virtual class $c_{M}(g,h)$.  Namely, for $(g,h) \in \sfS \times\sfS $, let $TM|_{M^{g,h}}= \bigoplus_{\lambda \in \Hom(H,S^1)} W_{\lambda}$ be the weight decomposition of the tangent bundle of $M$ restricted to $M^{g,h}$ where $H$ is the subgroup of $\sfS $ generated by $g$ and $h$. Then the {\bf obstruction bundle} on $[M^{g,h}/\sfS ]$ is the $\sfT$-equivariant vector bundle
\[
\calR_M(g,h)=\bigoplus_{a_{\lambda}(g)+a_{\lambda}(h)+a_{\lambda}((gh)^{-1})=2, \atop{\lambda\not=0}} W_{\lambda}
\]
and the corresponding virtual class $c_{M}(g,h)$ is defined as the $\sfT$-equivariant Euler class of $\calR_M(g,h)$
\[
c_{M}(g,h):=\be_{\sfT}\left(\calR_{M}(g,h)\right) \in H_{\sfT}(M^{g,h}).
\]
The associativity of this product follows immediately from the excess intersection formula and the projection formula, as in  \cite{GHK, JKK}:
%%%%%%%%%%%%%%%%%%%%%%%%%%%%%%%%%%%

\begin{thmx}[Theorem \ref{associativity}  below]
The ring $H_{orb,\sfR}([M/\sfS ])$ with the twisted product is associative.
\end{thmx}

%%%%%%%%%%%%%%%%%%%%%%%%%%%%%%%%%
To apply our GKM theorem to the computation of the equivariant Chen-Ruan cohomology, we introduce a new ring $\left(\calNH_{\sfR}(\nu [M/\sfS ]^{\sfR}),\star\right)$.  As a vector space,
\[
\calNH_{\sfR}(\nu [M/\sfS ]^{\sfR}):=\bigoplus_{g \in \sfS } H_{\sfT}(\nu(F^g\subset M)),
\]
where $F$ is the $\sfS$-invariant submanifold such that $[M/\sfS ]^{\sfR}=[F/\sfS ]$, and $\nu(F^g\subset M)$ denotes the normal bundle to $F^g$ in $M$.
The product $\star$  is defined in Section~\ref{star product} using the isotropy data for the $T$-action on each $\nu(F^g\subset M)$. We then prove the following two results.
%%%%%%%%%%%%%%%%%%%%%%%%%%%%%%%%%%%

\begin{thmx}[Theorem \ref{inertial associativity} below]
The ring $(\calNH_{\sfR}(\nu[M/\sfS ]^{\sfR}),\star)$ is associative.
\end{thmx}

%%%%%%%%%%%%%%%%%%%%%%%%%%%%%%%%%

%%%%%%%%%%%%%%%%%%%%%%%%%%%%%%%%%%%

\begin{thmx}[Theorem \ref{map of rings} below]
The natural restriction map 
\[
H_{orb,\sfR}([M/\sfS ]) \to \calNH_{\sfR}(\nu[M/\sfS ]^{\sfR})
\] 
is a homomorphism of graded associative rings.
\end{thmx}

%%%%%%%%%%%%%%%%%%%%%%%%%%%%%%%%%
\noindent Thus when this map is injective over $\Q$ or $\Z$, we may apply our orbifold GKM theorem to compute $H_{orb,\sfR}([M/\sfS ])$ using only local isotropy data at the fixed orbifold points.

In the final section, we compute the $\sfR$-equivariant Chen-Ruan cohomology for the toric orbifolds and conclude with two presentations of the ring, first as a quotient ring, and second as a subring of a direct sum of polynomial rings.

\subsection*{Acknowledgements} The authors graciously thank Rebecca Goldin, Eugene Lerman, Frank Moore, Reyer Sjamaar, and Ed Swartz for many helpful conversations throughout the course of this project.  The first author thanks the Mathematical Sciences Research Institute for its hospitality and support in Spring 2010. The second author  would like to express his gratitude to the Algebraic Structure and its Application Research Institute at KAIST for providing him an excellent research environment in 2011-2012.

%%%%%%%%%%%%%%%%%%%%%%%%%%%%%%%%%%
%%%%%%%%%%%%%%%%%%%%%%%%%%%%%%%%%%
\section{Local normal forms}\label{def 2.1}
In this section, we recall basics from \cite{LT} adapted to our setting.  Our main focus is to provide a local normal form for a Hamiltonian $\sfR$-orbifold $[M/\sfS]$. Generically to define the Chen-Ruan invariants, it is best to think of orbifolds as stacks.  As we are interested in global quotients, though, we do not need to use this full machinery.
%%%%%%%%%%%%%%%%%%%%%%%%%%%%%%%%%%
%%%%%%%%%%%%%%%%%%%%%%%%%%%%%%%%%%

%%%%%%%%%%%%%%%%%%%%%%%%%%%%%%%%%%
A {\bf locally free action} of a group $G$ on a space $X$ is one where the stabilizer subgroup $G_x$ of a point $x$ is finite for each $x\in X$. An {\bf orbifold} $[M/\sfS ]$ in this paper is defined as the quotient of a manifold $M$ by a locally free action of a compact torus $\sfS $. An orbifold chart for $[M/\sfS ]$ may be given as follows: let $x \in M$ and let $\sfS _x$ be the finite stabilizer of $x$ in $\sfS $. The group $\sfS _x$ acts on $T_xM$ and acts trivially on $T_x(\sfS\cdot x)$. Hence the orbifold chart around $[x/\sfS _x]$ is given by $\frac{T_xM}{T_x(\sfS x)}$ with the induced action of $\sfS _x$.  A {\bf symplectic orbifold} $([M/\sfS ], \omega)$ is an orbifold $[M/\sfS ]$ together with  an $\sfS $-invariant $2$-form $\omega \in \Omega^2(M)^{\sfS }$ such that for each $x \in M$, $\omega$ induces the $\sfS _x$-invariant $2$-form $\overline{\omega}$ on the chart $\frac{T_xM}{T_x(\sfS x)}$, which is closed at the origin and non-degenerate, i.e. the kernel of $\omega_x: T_xM \to T_x^*M$ is $T_x(\sfS x)$. A torus $\sfR$ {\bf acts on a symplectic orbifold} $([M/\sfS ],\omega)$ if there is a short exact sequence $\sfS  \inc \sfT \surj \sfR$ of compact tori, and the group $\sfT$ acts on $M$ extending the action of $\sfS $, and $\omega \in \Omega^2(M)^{\sfT}$ is a $\sfT$-invariant $2$-form.

%%%%%%%%%%%%%%%%%%%%%%%%%%%%%%%%%%
The $\sfR$-action on $([M/\sfS ],\omega)$ is {\bf Hamiltonian} if there is a $\sfT$-invariant map $\mu: M  \to \frakr^*$ where $\frakr:=\Lie\ \sfR$, satisfying the Hamiltonian condition, i.e. for each $\xi \in \ft$, the infinitesimal vector field $\xi_M$ and $d\mu^{\xi}$ are related by 
\[
\omega(\xi_M,-) = d\mu^{\xi}
\]
where $\mu^{\xi}:M\to \R$ is the component of the moment map given by $x \mapsto \lan\mu(x),\xi\ran$.  This is well-defined since $\frakr^*= \Ann \fraks \subset \ft^*$.

%%%%%%%%%%%%%%%%%%%%%%%%%%%%%%%%%%
For a function $f: M \to \R$ and a critical point $x \in M$ of $f$, the {\bf Hessian} of $f$ is the map $H(f)_x: T_xM \to T^*_xM$ defined by $w_x \mapsto d_x (\calL_wf)$, 
where $w$ is any local vector field around $x$ such that $w|_x=w_x$. Note that $H(f)_x$ does not depend on a local extension of $w_x$ to a vector field. 

A function $[f]: [M/\sfS ] \to \R$ on an orbifold $[M/\sfS ]$ is given by an $\sfS $-invariant function $f:M \to \R$.  An orbifold point is given by $[\sfS \cdot x/\sfS ]=[x/\sfS _x]$. An orbifold point $[x/\sfS _x] \in [M/\sfS ]$ is a {\bf critical orbifold point for $[f]$} if $d_xf=0$, i.e. $x$ is a critical point for $f$. This is a well-defined notion since $f$ is $\sfS $-invariant. Indeed if $d_xf=0$, then $\forall y\in \sfS \cdot x$, $d_yf=0$. Let $F:=\Crit(f) \subset M$ be the set of critical points for $f$. Then the $\sfS $-action preserves $F$ and therefore, if $F$ is a submanifold of $M$, then $[F/\sfS ]$ is a {\bf critical suborbifold }of $[M/\sfS ]$.  When extending Bott's version of Morse theory to orbifolds, a key hypothesis will be that the critical set is a suborbifold of $[M/\sfS]$.

Let $[f]: [M/\sfS ] \to \R$ be a function and $[x/\sfS _x] \in [M/\sfS ]$ a critical point.  Then $H(f)_x(w_x)=d_x(\calL_wf)=0$ for $w_x \in T_x(\sfS x)$, since  $\calL_w f=0$ on a neighborhood of $x$ by the $\sfS $-invariance of $f$. Hence $H(f)_x$ factors through
\[
H(f)_x: T_xM \to \frac{T_xM}{T_x(\sfS \cdot x)} \to T^*_xM.
\]
Furthermore, the image of $H(f)_x$ is contained in $\Ann T_x(\sfS  \cdot x)$. Thus we can define the induced {\bf orbifold Hessian} for $[f]$ as the $\sfS _x$-equivariant map
\[
\overline{H(f)_x}: \frac{T_xM}{T_x(\sfS \cdot x)} \to \left(\frac{T_xM}{T_x(\sfS \cdot  x)}\right)^* = \Ann T_x (\sfS  \cdot x).
\]
A critical suborbifold $[F/\sfS ]$ is {\bf non-degenerate} if $\forall x \in F$,
\[
\ker \overline{H(f)_x} = \frac{T_xF}{T_x(\sfS \cdot x)}.
\]
It is obvious that $[F/\sfS ]$ is non-degenerate if and only if $F$ is non-degenerate, since $F$ is non-degenerate when $\ker H(f)_x = T_xF$ for all $x\in F$.

Since our orbifold is presented as a global quotient by a torus, we may explicitly write down the slice theorem and the local normal form, following \cite{LT}. Let $[M/\sfS ]$ be a compact, connected Hamiltonian $\sfR$-orbifold where $\sfS  \inc \sfT \twoheadrightarrow \sfR$  is the torus extension such that $\sfT$ acts on $M$. The corresponding maps of Lie algebras are $\fraks \inc \ft \twoheadrightarrow \frakr$. We summarize what this implies as follows: 
\begin{enumerate}
\item the class $\omega\in \Omega^2(M)^{\sfS}$ is actually $\sfT$-invariant; 
\item the $\sfS $-action on $M$ is locally free; 
\item the map $\mu: M \to \frakr^*$ is a $\sfT$-invariant map; 
\item we have the Hamiltonian condition $\omega(\xi_M,-)=d\mu^{\xi}, \forall \xi \in \ft$; and
\item $\forall x \in M$, $\omega_x: T_xM \to T_x^*M$ induces an isomorphism
$$
\overline{\omega}_x: T_xM/T_x(\sfS \cdot x) \to (T_xM/T_x(\sfS \cdot x))^*=\Ann T_x(\sfS \cdot x).$$ 
\end{enumerate}
For each $x \in M$, let $\sfT_{\!x}$ denote the stabilizer in $\sfT$ of the point $x$, and $\ft_x:=\Lie(\sfT_{\!x})$ its Lie algebra.
%%%%%%%%%%%%%%%%%%%%%%%%%%%%%%%%%%
\begin{remark}
For a form $\alpha \in \Omega^*(M)^{\sfT}$ to induce an $\sfR$-equivariant form on the orbifold $[M/\sfS]$, it must be \emph{$\sfS$-basic}, i.e. $i_{X_{\xi}}\alpha =0$ for each $\xi \in \fraks$. If we define $\sfR$-equivariant forms on $[M/\sfS]$ as suggested in \cite{LM}, this coincides with the space of $\sfT$-equivariant forms on $M$ that are $\sfS$-basic.
\end{remark}
%%%%%%%%%%%%%%%%%%%%%%%%%%%%%%%%%%

We now turn to the linear algebra background needed to describe the  local normal form for a Hamitlonian $\sfR$-action on 
the orbifold $[M/\sfS]$.
%%%%%%%%%%%%%%%%%%%%%%%%%%%%%%%%%%
\begin{lemma}\label{basic linear algebra}
For every point $x \in M$, there is a (non-canonical) $\sfT_x$-equivariant  isomorphism 
\[
T_xM \cong \Ann\!(\fraks \oplus\ft_x) \oplus (T_x(\sfT x))^{\perp_{\omega_x}}\ ,
\]
where the $\sfT_x$-action on $\Ann\! (\fraks \oplus\ft_x)$ is the coadjoint action (and so is trivial, as $\sfT_x$ is abelian). Thus we have a $\sfT_x$-equivariant isomorphism
\[
\frac{T_xM}{T_x(\sfT x)} \cong \Ann\! (\fraks \oplus\ft_x) \oplus \frac{(T_x(\sfT x))^{\perp_{\omega_x}}}{T_x(\sfT x)}.
\]
\end{lemma}
%%%%%%%%%%%%%%%%%%%%%%%%%%%%%%%%%%
\begin{proof}
If $\xi \in \ft_x$, then the vector field $\xi_M|_x=0$, which implies that  $d_x\mu^{\xi}=0$ by property (4) above. Thus for any tangent vector $v_x$, we have $d_x\mu(v_x) \in \Ann \ft_x$ since $\lan d_x\mu(v_x),\xi\ran = \lan d_x\mu^{\xi}, v_x\ran$ for every $ \xi \in \ft_x$. Therefore the image of $T_xM$ under the map $d_x\mu$ lies in $$\Ann \fraks \cap \Ann \ft_x=\Ann(\fs\oplus \ft_x).$$ The map $d_x\mu$ is actually a surjection. 
Using property (4), $d_x\mu^{\xi} = 0$ implies $\xi_M|_x \in T_x(\sfS x)$. The dual map is $(d_x\mu)^*:\ft/(\fraks \oplus \ft_x) \to T_x^*M$.  Let $ \tilde{\ft}_x:=\{ \xi \in \ft \ \ |\ \ \xi_M|_x \in T_x(\sfS x)\}$ which is the Lie algebra of $\tilde{\sfT}_x:=\{t \in \sfT \ \ |\ \ tx \in \sfS x\}$ and then it follows that $\ker (d_x\mu)^* = \tilde{\ft}_x/ (\fraks \oplus \ft_x)$. We can show that $\tilde{\sfT}_x=\sfS \cdot \sfT_x$ and so $\tilde{\ft}_x=\fraks \oplus \ft_x$. In fact, 
if $st \in \sfS \cdot \sfT_x$ then $ stx \in \sfS x$ and conversely if $t x \in \sfS x$ (i.e. $tx=sx$ for some $s\in \sfS $), then $t^{-1}s \in \sfT_x$ and so $t \in \sfS\cdot \sfT_x$.
Thus $d_x\mu: T_xM \to \Ann (\fraks \oplus\ft_x)$ is surjective. By definition, 
$\ker (d_x\mu) = \{ v \in T_xM \ |\ \omega_x(\xi_M|_x, v)=0, \forall \xi \in \ft\}=(T_x(\sfT x))^{\perp_{\omega_x}} $. By choosing a $\sfT_x$-invariant splitting of the surjection $d_x\mu$, we have the isomorphism. The second claim follows from $ T_x(\sfT x) \subset (T_x(\sfT x))^{\perp_{\omega_x}}$. Indeed, $\omega_x(\xi_M|_x,\eta_M|_x)=\lan d_x\mu^{\xi}, \eta_M|_x\ran = \eta_M(\mu^{\xi})|_x=0$ for each  $\xi,\eta \in \ft$ where the last equality follows since $\mu^{\xi}:M \to \R$ is constant along $\sfT x$ and $\eta_M|_x \in T_x(\sfT x)$.
\qed\end{proof}
%%%%%%%%%%%%%%%%%%%%%%%%%%%%%%%%%%

The following results are applications of standard results from the theory of Lie group actions.  
They are discussed further in \cite{LT}, and we have adapted them to our setting here.
%%%%%%%%%%%%%%%%%%%%%%%%%%%%%%%%%%
\begin{theorem}[Slice theorem for the $\sfT$-action on $M$] 
\label{thm:slice}
For each point $x \in M$, there is a $\sfT$-invariant neighborhood $U$ of the orbit $\sfT \cdot x$ 
and a $\sfT$-equivariant isomorphism
\[
U \cong \sfT \times_{\sfT_x} \frac{T_xM}{T_x(\sfT x)},
\]
where $\sfT_x$ acts on $\sfT \times \frac{T_xM}{T_x(\sfT x)}$ by $(t,[v_x])\mapsto(ts^{-1},[s_*v_x])$.
\end{theorem}
%%%%%%%%%%%%%%%%%%%%%%%%%%%%%%%%%%
As a corollary of Lemma~\ref{basic linear algebra} and the Slice Theorem~\ref{thm:slice}, we obtain
%%%%%%%%%%%%%%%%%%%%%%%%%%%%%%%%%%
\begin{cor}[Local  normal form for the $\sfT$-action on $M$]\label{cor:normal form}
For each point $x \in M$, there is a $\sfT$-invariant neighborhood $U$ of $\sfT\cdot x$ such that we have a $\sfT$-equivariant isomorphism
\[
\Psi: U \cong \sfT \times_{\sfT_x} \left( \Ann (\fraks \oplus\ft_x) \oplus \frac{(T_x(\sfT x))^{\perp_{\omega_x}}}{T_x(\sfT x)}\right).
\]
\end{cor}
%%%%%%%%%%%%%%%%%%%%%%%%%%%%%%%%%%
Both Theorem~\ref{thm:slice} and Corollary~\ref{cor:normal form} may be rephrased as an orbifold slice theorem and an orbifold local normal form theorem, as in \cite{LT}.
%%%%%%%%%%%%%%%%%%%%%%%%%%%%%%%%%%
\begin{theorem}[Orbifold slice theorem for the $\sfR$-action on {$[M/\sfS ]$}] 
For every point $[x/\sfS _x] \in [M/\sfS ]$, there is an $\sfR$-invariant orbifold neighborhood $[U/\sfS ]$ of $[x/\sfS_x ]$ such that we have an $\sfR$-equivariant isomorphism
$$[U/\sfS ] \cong [(\sfT \times_{\sfT_x} W)/\sfS ]=\sfR \times_{\sfR_{[x/\sfS _x]}} [W/\sfS _x],$$ 
where $\sfR_{[x/\sfS _x]} := \sfT_x/\sfS _x$ and $W:=\frac{T_xM}{T_x(\sfT x)}$.
\end{theorem}
%%%%%%%%%%%%%%%%%%%%%%%%%%%%%%%%%%
\begin{cor}[Orbifold local normal form for the $\sfR$-action on {$[M/\sfS]$}]
\label{cor:orbifold lnf}
For each $[x/\sfS _x] \in [M/\sfS ]$, there is an $\sfR$-invariant orbifold neighborhood $[U/\sfS ]$ of $[M/\sfS ]$ such that we have an $\sfR$-equivariant isomorphism
\[
 [U/\sfS ] \cong \sfR \times_{\sfR_{[x/\sfS _x]} }\left(T_1^*(\sfR/\sfR_{[x/\sfS _x]}) \oplus \left[\left.\frac{(T_x(\sfT x))^{\perp_{\omega_x}}}{T_x(\sfT x)}\right/  \sfS _x  \right]\right)
\]
where $\sfR_{[x/\sfS _x]} = \sfT_x/\sfS _x$.
\end{cor}
%%%%%%%%%%%%%%%%%%%%%%%%%%%%%%%%%%
\begin{remark}
We can define and derive everything in this section by letting $\sfS $ be a subgroup of $\sfT$, not necessarily a (connected) subtorus.
\end{remark}

%%%%%%%%%%%%%%%%%%%%%%%%%%%%%%%%%%
%%%%%%%%%%%%%%%%%%%%%%%%%%%%%%%%%%
%%%%%%%%%%%%%%%%%%%%%%%%%%%%%%%%%%
\section{Orbifold fixed points and the orbifold 1-skeleton}
In this section, we use the local normal form of Corollary~\ref{cor:orbifold lnf} to determine the {\bf orbifold fixed points} and {\bf orbifold 1-skeleton} of $[M/S]$.
%%%%%%%%%%%%%%%%%%%%%%%%%%%%%%%%%%
%%%%%%%%%%%%%%%%%%%%%%%%%%%%%%%%%%
%%%%%%%%%%%%%%%%%%%%%%%%%%%%%%%%%%

%%%%%%%%%%%%%%%%%%%%%%%%%%%%%%%%%%
%%%%%%%%%%%%%%%%%%%%%%%%%%%%%%%%%%
\subsection{Orbifold fixed points $[F/\sfS ]$}\label{se:fixed}
The set $[M/\sfS]^{\sfR}$ of fixed orbifold points in $[M/\sfS ]$ with respect to the $\sfR$-action is a symplectic  suborbifold, cf. \cite[p. 4206]{LT}. If we let $F:=\{ x \in M\ |\ \sfT x = \sfS x\}$, then $[F/\sfS] = [M/\sfS]^{\sfR}$. Moreover, since $\sfS$ acts on $M$ locally freely, we have that $x\in F$ if and only if $\ft/(\fs\oplus \ft_x) = 0$. Thus the $\sfT$-equivariant local normal form at a point $x\in F$ becomes $U_x\cong \sfT\times_{\sfT_x} W$ where $W:=\frac{T_xM}{T_x(\sfT x)}$. In this section, we compute $F$ explicitly by using this normal form.
%%%%%%%%%%%%%%%%%%%%%%%%%%%%%%%%%%
%%%%%%%%%%%%%%%%%%%%%%%%%%%%%%%%%%

Let $\sfT_{x,1}$ be the connected component of $\sfT_x$ containing the identity element. Then $\sfT_x=\sfT_{x,1}\sfS _x$. In fact, for all $t_x \in \sfT_x$, there are $t_{x,1}\in \sfT_{x,1}$ and $s \in \sfS$ such that $t_x=t_{x,1}s$ since $\sfT=\sfT_{x,1} \sfS $. Therefore $s \in \sfT_x \cap \sfS = \sfS_x$.
\begin{lemma}\label{technical lemma}
For $\alpha \in \Hom(\sfT_x,S^1)$, define $\overline \alpha$ to be the induced map $\sfT_x/\sfS _x \to S^1/\alpha(\sfS _x)$. Then for $\alpha,\beta \in \Hom(\sfT_x,S^1)$, 
\[
\alpha|_{\sfT_{x,1}} = \beta|_{\sfT_{x,1}} \ \ \Leftrightarrow \ \ \overline{\alpha-\beta}=0.
\]
\end{lemma}

\proof Since $\sfT_x=\sfT_{x,1}\cdot \sfS _x$, every element $t \in \sfT_x$ can be written as  
$t=t_1\cdot s$ for some element $t_1\in \sfT_{x,1}$ and $s\in  \sfS _x$. Thus, if $\alpha|_{\sfT_{x,1}} = \beta|_{\sfT_{x,1}}$, we have 
$$\alpha(t)(\beta(t))^{-1} = \alpha(s)(\beta(s))^{-1} \in (\alpha-\beta)(\sfS _x).$$ 
On the other hand, if $(\alpha-\beta)(\sfT_{x,1}) \subset (\alpha-\beta)(\sfS _x)$, 
then $(\alpha-\beta)(\sfT_{x,1})=1$ since $\sfT_{x,1}$ is connected and $\sfS _x$ 
is discrete. \qed

Recall that in the local normal form, we may identify  $W:=\frac{T_xM}{T_x(\sfT x)}=\frac{(T_x(\sfT x))^{\perp_{\omega_x}}}{T_x(\sfT x)}$.  This
vector space is equipped with a $\sfT_{x,1}$-action. Let
$$\displaystyle{W=\bigoplus_{\lambda \in \Hom(\sfT_{x,1},S^1)} \overline W_{\lambda}}$$ 
be the weight decomposition of the $\sfT_{x,1}$-action on $W$.
%&&%%%%%%%%%%%%%%%%%%%%%%%%%%%%%%%%%%%%%%

\begin{lemma}\label{chartF}
Under the $\sfT$-equivariant local normal form isomorphism $\Psi$, for each point $x\in F$, we have 
$$F \cap U_x \cong \sfT \times_{\sfT_x} \overline W_{0}.$$ 
In particular, $F$ is a manifold.
\end{lemma}
\proof Recall that $\sfT$ acts on $[r,v] \in \sfT \times_{\sfT_x} W$ by $t\cdot [r,v]=[tr,v]$. Moreover, two equivalence classes $[r_1,v_1]=[r_2,v_2]$ are equal if and only if there is some $t_x \in \sfT_x$ such that $(r_2t_x^{-1},t_{x*}v_2)=(r_1,v_1)$, or equivalently if and only if $r_2r_1^{-1} \in \sfT_x$ and $(r_2r_1^{-1})_*v_2=v_1$. Therefore
\begin{eqnarray*}
&& \hskip -0.5in [r,v]  \mbox{ is in the image of }F\cap U_x \mbox{ under } \Psi \\ &\Leftrightarrow& \forall t \in \sfT,  \exists s \in \sfS , \ \  \mbox{such that} \ \ t\cdot[r,v] = s\cdot [r,v],  \ \ \ \mbox{by definition of } F\\ 
&\Leftrightarrow& \forall t \in \sfT,  \exists s \in \sfS , \ \  \mbox{such that} \ \ st^{-1} \in \sfT_x \mbox{ and } (st^{-1})_*v=v \\
&\Leftrightarrow& \forall t_x \in \sfT_x,  \exists s_x \in \sfS _x, \ \  \mbox{such that} \ \ (t_xs_x)_*v=v .
\end{eqnarray*}
For the $(\Leftarrow)$ of the last equivalence, we observe that for a given element $t \in \sfT$, the element $s$  in  $\sfS $ such that $s^{-1}t \in \sfT_x$ is unique up to $\sfS _x=\sfS \cap \sfT_x$.
Let 
\[
W_F:=\{v \in W\ |\ \forall t_x \in \sfT_x,  \exists s_x \in \sfS _x \ \mbox{ such that }\ (t_xs_x)_*v=v \}.
\] Let $W=\bigoplus_{\alpha \in \Hom(\sfT_x,S^1)} W_{\alpha}$. Then we can also write as $W_F = \bigoplus_{\alpha \in \Hom(\sfT_x,S^1)} W_{\alpha}$, where the direct sum runs over $\alpha \in \Hom(\sfT_x,S^1)$ such that for each $t_x \in \sfT_x$, there is an $s_x \in \sfS _x$ such that $\alpha(t_xs_x)=1$. However this condition is equivalent to $\alpha(\sfT_x)=\alpha(\sfS_x)$, that is, $\overline{\alpha}=0$. Thus by Lemma~\ref{technical lemma}, $W_F=\overline W_0$, as desired.
\qed
%&&%%%%%%%%%%%%%%%%%%%%%%%%%%%%%%%%%%%%%%
\begin{remark}\label{rem 3.3}
We may also write $F$ as
\begin{eqnarray*}
F & = & \left\{x \in M \ \Big|\ \xi_M|_x \in T_x(\sfS x), \forall \xi \in \ft\right\} \\
&  = &\left\{x \in M \ \left|\ [\xi_M|_x]=0 \mbox{ in } \frac{T_xM}{T_x(\sfS x)}, \forall \xi \in \ft\right.\right\}.
\end{eqnarray*}
\end{remark}
%%%%%%%%%%%%%%%%%%%%%%%%%%%%%%%%%%
We will need the following technical lemma to prove Theorem~\ref{thm:injectivity}.

\begin{lemma}\label{global isotropy}
 Let $y \in F \cap U_x$ for $x \in F$, then $\sfT_y=\bigcap_{v_{\alpha}\not=0} \ker \alpha$, 
 where  $[t,v] \in \sfT \times_{\sfT_x} \overline W_{0}$ corresponds to $y$ under the local 
 normal form $\Psi$, and $v_{\alpha}$ is the $\alpha$-component of $v$ in the weight 
 decomposition $$W=\bigoplus_{\alpha \in \Hom(\sfT_x,S^1)} W_{\alpha}.$$ 
 In particular, $\sfT_{x,1}=\sfT_{x',1}$ for all $x'$ in the connected component of $F$ containing $x$.
\end{lemma}
\proof  It suffices to show this in the case when $v=v_{\alpha}\not=0$. If $r \in \sfT_y$, then $r\cdot [t,v]=[rt,v]=[t,v]$. Therefore $r \in \sfT_x$ and $\alpha(r)=1$, i.e. $r \in \ker \alpha$. If $r\in \ker \alpha \subset \sfT_x$, then $r\cdot[t,v]=[rt,v]=[t,\alpha(r)v]=[t,v]$, so $r \in \sfT_y$. This proves the first claim. 

Since $\alpha|_{\sfT_{x,1}}=0$, $\sfT_{x,1}$ must be contained in every $\sfT_y$, where $y$ is in the neighborhood $U_x$. For any $x,x' \in F$ such that $\exists y \in U_x \cap U_{x'} $, $\sfT_{x,1}$ and $\sfT_{x',1}$ are the connected component of $\sfT_y$ containg $1$, so they coincide. This proves the latter claim. \qed
%%%%%%%%%%%%%%%%%%%%%%%%%%%%%%%%%%
%%%%%%%%%%%%%%%%%%%%%%%%%%%%%%%%%%
%%%%%%%%%%%%%%%%%%%%%%%%%%%%%%%%%%
%%%%%%%%%%%%%%%%%%%%%%%%%%%%%%%%%%
%%%%%%%%%%%%%%%%%%%%%%%%%%%%%%%%%%
%%%%%%%%%%%%%%%%%%%%%%%%%%%%%%%%%%
%%%%%%%%%%%%%%%%%%%%%%%%%%%%%%%%%%
%%%%%%%%%%%%%%%%%%%%%%%%%%%%%%%%%%
%%%%%%%%%%%%%%%%%%%%%%%%%%%%%%%%%%
%%%%%%%%%%%%%%%%%%%%%%%%%%%%%%%%%%
\subsection{Orbifold 1-skeleton $[M_1/\sfS ]$}
%%%%%%%%%%%%%%%%%%%%%%%%%%%%%%%%%%
The union of $1$-dimensional orbits in $[M/\sfS ]$ corresponds to the union of 
$(\dim \sfS  +1)$-dimensional orbits in $M$: let $M_1^{\circ}$ be the set of all points 
$x \in M$ such that $\sfT x$ is $(\dim \sfS  + 1)$-dimensional, so that $[M_1^{\circ}/\sfS ]$ 
is the union of $1$-dimensional orbits. The {\bf orbifold $1$-skeleton} is by definition 
$[(M_1^{\circ} \cup F) /\sfS ]$. In this section, we calculate $M_1^{\circ}$ explicitly 
using the normal form and
show that the closure $M_1$ of $M_1^{\circ}$ is $M_1^{\circ} \cup F$. 
We then demonstrate that the closure $N$ of a connected component 
$N^{\circ}$ of $M_1^{\circ}$ is a manifold and $[N/\sfS ]$ is a 
Hamiltonian $\sfR$-orbifold. Since $x\in M_1^{\circ}$ if and only if  $\ft/(\fs\oplus\ft_x)$ is 
one-dimensional, the $\sfT$-equivariant isomorphism $\Psi$ becomes $U_x \cong \sfT \times_{\sfT_x} \left( \Ann (\fs \oplus\ft_x) \oplus W\right)$ where $W:= \frac{(T_x(\sfT x))^{\perp_{\omega_x}}}{T_x(\sfT x)}$. 
We recall that $$W=\bigoplus_{\lambda \in \Hom(\sfT_{x,1},\Z)} \overline W_{\lambda}=\bigoplus_{\alpha \in \Hom(\sfT_x,\Z)} W_{\alpha}$$ 
are the weight decompositions with respect to the actions of $\sfT_{x,1}$ and $\sfT_x$ respectively.  Note that the notation $W_\alpha$ denotes a $\sfT_x$ representation, whereas $\overline W_\lambda$ denotes a representation of its identity component  $\sfT_{x,1}$.

\begin{lemma}
For $x \in M_1^{\circ}$,  the isomphism $\Psi$ induces 
\[
M_1^{\circ} \cap U_x \cong \sfT \times_{\sfT_x} \left( \Ann (\fraks \oplus\ft_x) \oplus \overline W_0\right).
\]
In particular, $[M_1^{\circ}/\sfS ]$ is a Hamiltonian $\sfR$-orbifold.
\end{lemma}
\proof If $v \in W_{\alpha}$ with $\alpha|_{\sfT_{x,1}}\not=0$, then $\sfT_x \cdot v$ is at least one dimensional. Therefore $\sfT\cdot [t,a,v]= \sfT \times_{\sfT_x} (\{a\} \oplus \sfT_x v)$ is at least $(\dim \sfS  + 2)$-dimensional. If $v \in W_{\alpha}$ with $\alpha|_{\sfT_{x,1}}=0$, then $\sfT\cdot [t,a,v] = \sfT \times_{\sfT_x} (a \oplus v) = \sfT \times_{\sfT_x} (\{a\} \oplus\alpha(\sfT_x)\cdot v)$ is exactly $(\dim \sfS  + 1)$-dimensional since $\alpha(\sfT_x)$ is discrete. The second claim follows from the fact that the symplectic structure of $W$ restricts to the symplectic structure on $\overline W_0$. \qed
%%%%%%%%%%%%%%%%%%%%%%%%%%%%%%%%%%
\begin{lemma}\label{comp 1 sk}
Let $M_1$ to be the closure of $M_1^{\circ}$ in $M$. Then $M_1 = M_1^{\circ} \cup F$.
\end{lemma}
\proof For each $x \in F$, every neighborhood of $x$ in $U_x$ contains a point in $M_1^{\circ}$, namely $[t,v]$ where $v \in W_{\alpha}$ where $\overline{\alpha}\not=0$. And so the orbit of $[t,v]$ is ($\dim \sfS $ +1)-dimensional. Indeed $\sfT\cdot [t,v]=\sfT \times_{\sfT_x} \alpha(\sfT_x) v$. Thus $F \subset M_1$.  If $x \in M_1\backslash(M_1^{\circ} \cup F)$, then the $\sfT$-orbit of $x$ is more than $(\dim \sfS  +1)$-dimensional. However, if $y \in M$ has an $m$-dimensional $\sfT$-orbit, then for every point in $U_y$, its $\sfT$-orbit is at least $m$-dimensional. This leads to a contradiction, and so we must have $M_1 \subset (M_1^{\circ} \cup F)$. \qed
%%%%%%%%%%%%%%%%%%%%%%%%%%%%%%%%%%
\begin{remark}\label{normal form for 1 skeleton}
For $x \in F$, let $N$ be the closure of a connected component of $M_1^{\circ}$ such that $x \in N$, then  $N \cap U_x \cong \sfT \times_{\sfT_x} \left(\overline W_{0}\oplus \overline W_{\lambda}\right)$ for some $\alpha\not=0 \in \Hom(\sfT_{x,1},S^1)$. This implies that $[N/\sfS ]$ is a Hamiltonian $\sfR$-orbifold. 
\end{remark}
%%%%%%%%%%%%%%%%%%%%%%%%%%%%%%%%%%
%%%%%%%%%%%%%%%%%%%%%%%%%%%%%%%%%%
%%%%%%%%%%%%%%%%%%%%%%%%%%%%%%%%%%
%%%%%%%%%%%%%%%%%%%%%%%%%%%%%%%%%%
%%%%%%%%%%%%%%%%%%%%%%%%%%%%%%%%%%
%%%%%%%%%%%%%%%%%%%%%%%%%%%%%%%%%%
%%%%%%%%%%%%%%%%%%%%%%%%%%%%%%%%%%
%%%%%%%%%%%%%%%%%%%%%%%%%%%%%%%%%%
%%%%%%%%%%%%%%%%%%%%%%%%%%%%%%%%%%
%%%%%%%%%%%%%%%%%%%%%%%%%%%%%%%%%%
%%%%%%%%%%%%%%%%%%%%%%%%%%%%%%%%%%
%%%%%%%%%%%%%%%%%%%%%%%%%%%%%%%%%%
%%%%%%%%%%%%%%%%%%%%%%%%%%%%%%%%%%
%%%%%%%%%%%%%%%%%%%%%%%%%%%%%%%%%%
\section{The Morse-Bott property and Injectivity}\label{sec:MB}
We continue to use the notation from Section \ref{def 2.1}: $[M/\sfS ]$ is a Hamiltonian $\sfR$-orbifold with an $\sfR$-invariant moment map $[\mu]: [M/\sfS ] \to \frakr^*$. Let $\xi \in \ft$ be a rational element and let $\sfR_1 \subset \sfR$ be the image of $\sfT_1:=\overline{\{\exp(t\xi), t \in \R\}} \subset \sfT$ in $\sfR$. In this section, we show that $\mu^{\xi}: M \to \R$ is Morse-Bott for every rational $\xi \in \ft$ such that $\dim \sfR_1=n+1$. This naturally leads to the injectivity theorem for compact Hamiltonian $\sfR$-orbifolds.  
%%%%%%%%%%%%%%%%%%%%%%%%%%%%%%%%%%
\begin{lemma}\label{holm-matsumura}
For any rational element $\xi$ in $\mathfrak{t}$, the map $\mu^{\xi}: M \to \R$ is a Morse-Bott function. 
\end{lemma}
\proof In order to show that $\Crit(\mu^{\xi})$ is a submanifold of $M$, by Lemma \ref{chartF}, it 
suffices to prove that $\Crit(\mu^{\xi})$ coincides with the submanifold which yields the $\sfR_1$-
fixed suborbifold in $[M/\sfS]$. That is, following Remark~\ref{rem 3.3}, we must show that 
\[
Q:=\Crit(\mu^{\xi}) = \{ x \in M \ |\ \xi_M|_x \in T_x(\sfS x)\}.
\] 
By definition, $\Crit(\mu^{\xi}) = \{ x \in M \ |\ d_x\mu^{\xi} =0 \}$. The equation $d_x\mu^{\xi}=0$ implies that $[\xi_M|_x]=\overline{\omega_x}^{-1}(d_x \mu^{\xi})=0$ in $\frac{T_xM}{T_x(\sfS x)}$.  Therefore the claim follows.  

Now we turn to non-degneracy.  The function $\mu^\xi$ is non-degenerate if its Hessian satisfies
\[
\ker H(\mu^{\xi})_x = T_xQ.
\] 
The Hessian is evaluated at a tangent vector $v_x \in T_xQ$ by
\begin{eqnarray*}
H(\mu^{\xi})_x(v_x)  \ =\  d_x(\calL_v\mu^{\xi}) \ = \ \calL_v(d\mu^{\xi})|_x \ = \ 0,
 \end{eqnarray*} 
since $d\mu^{\xi}|_Q=0$ and $v_x \in T_xQ$.  Thus $\ker H(\mu^{\xi})_x \supset T_xQ$. On the other hand, we may identify
\begin{eqnarray*}
T_xQ & = & \{ v_x \in T_xM \ |\ [\xi_M, v]|_x \in T_x(\sfS x)\}\\
 & = & \{ v_x \in T_xM \ |\ [\overline{\xi_M}, \overline{v}]|_x=0\},
 \end{eqnarray*}
 where $\overline{\xi_M}$ and $\overline{v}$ are the local vector fields induced on 
 $\frac{T_xM}{T_x(\sfS x)}$. Since $\overline{\xi_M}=\overline{\omega}^{-1}(d\mu^{\xi})$, 
 we have
 \begin{eqnarray*}
 [\overline{\xi_M}, \overline{v}]|_x & = &  
 (\calL_{\overline{v}}\overline{\omega}^{-1})|_x(d_x\mu^{\xi}) +
  \overline{\omega}^{-1}_x(\calL_{\overline{v}}(d\mu^{\xi})|_x)\\
  & = &  \overline{\omega}^{-1}_x(H(\mu^{\xi})_x(v_x)).
  \end{eqnarray*} 
  Therefore if $v_x \in \ker H(\mu^{\xi})_x$, then $[\overline{\xi_M}, \overline{v}]|_x=0$, i.e. $\ker H(\mu^{\xi})_x \subset T_xQ$. 
 \qed
%%%%%%%%%%%%%%%%%%%%%%%%%%%%%%%%%%´
\begin{remark}  
In particular, $[\mu^{\xi}]: [M/\sfS ] \to \R$ is an orbifold Morse-Bott function. By this, we mean that $\Crit([\mu^{\xi}])=[M/\sfS ]^{\sfR_1}$ is an orbifold,  and $[\mu^{\xi}]$ is non-degenerate.
\end{remark}

\begin{remark}
A further discussion of the local normal form for symplectic toric orbifolds may be found in 
\cite[Section 2]{GHH}. The ideas developed there motivated this project.
In the case when $M$ is the level set for an $\sfS$-moment map on a Hamiltonian $\sfT$-space $X$, 
our Lemma~\ref{holm-matsumura} is exactly \cite[Lemma 2.2]{GHH}.    The remarks in the footnote 
following the local normal form theorem in  \cite[Section 2]{GHH} also apply here.   
\end{remark}
%%%%%%%%%%%%%%%%%%%%%%%%%%%%%%%%%%
%%%%%%%%%%%%%%%%%%%%%%%%%%%%%%%%%%
The following two lemmata are well-known.
\begin{lemma}[Atiyah-Bott over $\Q$, c.f.\ Lemma 7.1 \cite{TW2}]\label{atiyah-bott}
Let $E \to F$ be a $\sfT$-equivariant complex vector bundle over a connected $\sfT$-space $F$. Let $D$ and $S$ be the $\sfT$-equivariant disk and sphere bundle corresponding $E$ with a choice of a $\sfT$-invariant metric. Assume that 
\begin{itemize}
\item[$(\Q 1)$] there is a subtorus $\sfT_1 \subset \sfT$ that fixes $F$ pointwise, and acts non-trivially on the fibers. 
\end{itemize}
Then the $\sfT$-equivariant Euler class $\be_{\sfT}(E,\Q)$ is a non-zero divsor. In particular, we have the short exact sequence over $\Q$:
\[
\xymatrix{
0 \ar[r] & H^i_{\sfT}(D,S; \Q)\ar[d]_{\cong_{Thom}}  \ar[r] & H^i_{\sfT}(D; \Q)\ar[d]_{\cong_{Homot}} \ar[r] & H_{\sfT}^i(S; \Q) \ar[r] & 0\\
& H_{\sfT}^{i-\lambda}(F; \Q) \ar[r]_{\cup \be_{\sfT}(E)}& H^i_{\sfT}(F; \Q) &
}
\]
The same claim holds over $\Z$ under a stronger assumption on the action of $\sfT_1$:
\begin{itemize}
\item[$(\Z 1)$] there is a subtorus $\sfT_1 \subset \sfT$ that fixes $F$ pointwise, and each weight $$\lambda \in \Hom(\sfT_1,S^1) \cong \Z^r$$ is primitive whenever $W_\lambda\neq\emptyset$. % ($r\geq 1$). 
\end{itemize}
\end{lemma}
Here recall that an element $\lambda$ of a $\Z$-module $R$ is \emph{primitive} if and only if $\lambda=a\lambda'$ for $a \in \Z$ and $m' \in R$ implies $a=\pm 1$.  

\begin{example}\label{ex:torsion}
If axiom $(\Z 1)$ is not satisfied, then the Euler class $e_{\sfT}(E,\Z)$ may be a torsion class in equivariant 
cohomology with integer coefficients.
Let $F$ be the Klein bottle, or any other space with $2$-torsion in its integral cohomology ring.  
Consider the trivial bundle $E=\C\times F\to F$, equipped with 
the circle action $t\cdot (z,f) = (t^2\cdot z,f)$, for $z\in \C$ and $f\in F$, which fixes the base pointwise and spins 
each fiber at double speed.  Then $H_{S^1}^*(F;\Z) = H^*(F;\Z)\otimes \Z[x]$ and under this identification we have
$e_{\sfT}(E,\Z)=1\otimes 2x$.
Let $\alpha\in H^*(F;\Z)$ be a $2$-torsion class (i.e.\ $2\alpha=0$),  then $\alpha\otimes 1$ is an 
equivariant class, and 
\begin{eqnarray*}
\alpha\otimes 1\cup e_{\sfT}(E,\Z) & = & \alpha\otimes 1\cup 1\otimes 2x\\
& = & \alpha \otimes 2x \\
& = & 2\alpha\otimes x\\
%& = & 0\otimes x \\
& = & 0.
\end{eqnarray*}
Thus,  $e_{\sfT}(E,\Z)$ is a zero divisor.
\end{example}

%%%%%%%%%%%%%%%%%%%%%%%%%%%%%%%%%%
\begin{lemma}[c.f.\ Lemma 4.4. \cite{LT}]\label{morse-bott}
Let $f: M\to \R$ be a $\sfT$-invariant Morse-Bott function on a compact manifold $M$ equipped with an action of $\sfT$. 
Choose $[a,b] \subset \R$ which contains a unique critical value $c$. Let $F$ be the critical submanifold such that 
$f(F)=c$. Let $E^-$ be its $\sfT$-equivariant negative normal bundle and $D$ and $S$ be the corresponding 
$\sfT$-equivariant disk and sphere bundles. Then $(M_b^-,M_a^-)$ can be retracted onto the pair $(D,S)$ so 
that we have
\[
\cong_{MB}:H^*_{\sfT}(M_b^-,M_a^-;\Z) \cong H^*_{\sfT}(D,S;\Z)
\]
where $M_a^-:=f^{-1}(-\infty,a)$.
\end{lemma}

%%%%%%%%%%%%%%%%%%%%%%%%%%%%%%%%%%
%%%%%%%%%%%%%%%%%%%%%%%%%%%%%%%%%%
By Lemma \ref{holm-matsumura},  the following generalization of \cite[Proposition 2.1]{TW}  obviously follows from Lemmata \ref{atiyah-bott} and \ref{morse-bott}.
\begin{prop}\label{s.e.s}
Let $c \in \R$ be a critical value for $\mu^{\xi}$ and let $F_c$ be the set of critical points contained in $(\mu^{\xi})^{-1}(c)$. Assume that 
\begin{itemize}
\item[$(\Q 2)$] for each connected component $F_c'$ of $F_c$, there is a subtorus of $\sfT$, that fixes $F_c'$ pointwise and acts non-trivially on the negative normal bundle $E_c^-|_{F_c'}$. 
\end{itemize}
Let $\epsilon\geq 0$ such that $c$ is the only critical point in $[a,b]:=[c-\epsilon, c+\epsilon]$. Then we have the short exact sequence over $\Q$:
\[
\xymatrix{
0 \ar[r] & H^i_{\sfT}(M_b^-,M_a^-;\Q) \ar[r]\ar[d]_{\cong_{MB}} & H^i_{\sfT}(M_b^-;\Q) \ar[r] \ar[d] & H^i_{\sfT}(M_a^-;\Q) \ar[r] & 0 \\
& H^i_{\sfT}(D_c,S_c;\Q)\ar[d]_{\cong_{Thom}}  \ar[r] & H^i_{\sfT}(D_c;\Q)\ar[d]_{\cong_{Homot}}  &&\\
& H_{\sfT}^{i-\lambda}(F_c;\Q) \ar[r]_{\cup \be_{\sfT}(E_c^-)}& H^i_{\sfT}(F_c;\Q) &
}
\]
where $E_c^-$ is the negative normal bundle for $F_c$ and $D_c$ and $S_c$ are the corresponding $\sfT$-equivariant disk and sphere bundles.

The claim also holds over $\Z$ if we assume
\begin{itemize}
\item[$(\Z 2)$] for each connected component $F_c'$ of $F_c$, there is a subtorus of $\sfT$ that fixes $F_c'$ pointwise and acts on the negative normal bundle $E_c^-|_{F_c'}$ in such a way that each weight $\lambda \in \Hom(\sfT_1,S^1) \cong \Z^r$ is primitive ($r\geq 1$).  
\end{itemize}
\end{prop}
%%%%%%%%%%%%%%%%%%%%%%%%%%%%%%%%%%

\begin{remark}
Note that the hypothesis $(\Q 2)$ implies $(\Q 1)$, and $(\Z 2)$ implies $(\Z 1)$.
\end{remark}

\begin{remark}
We can always choose a rational element $\xi$ in $\mathfrak{t}$ such that 
\[
\{ x \in M \ |\ \xi_M|_x \in T_x(\sfS x)\} = \{ x \in M \ |\ \xi'_M|_x \in T_x(\sfS x), \forall \xi' \in \ft\},
\] 
or equivalently, $[M/\sfS ]^{\sfR_1} =[M/\sfS ]^{\sfR}.$
\end{remark}
%%%%%%%%%%%%%%%%%%%%%%%%%%%%%%%%%%
\begin{theorem}\label{thm:injectivity}
Let $F:=\{ x \in M \ |\ \sfT x= \sfS x\}$ so that $[M/\sfS ]^{\sfR} = [F/\sfS ]$, and let $i: F \inc M$ be the inclusion map. 
Then the induced map $i^*: H^*_{\sfT}(M,\Q) \inc H^*_{\sfT}(F,\Q)$ is injective.
Equivalently, the induced map $i^*: H^*_{\sfR}([M/\sfS ],\Q) \inc H^*_{\sfR}([F/\sfS ],\Q)$ is injective.

Moreover, if
\begin{itemize}
\item[$(\Z 3)$] for each connected component $F'$ of $F$, each weight of the action of $\sfT_{F'}$ on the
normal bundle $\nu(F'\subset M)$ is primitive
\end{itemize}
then $i^*: H^*_{\sfT}(M,\Z) \inc H^*_{\sfT}(F,\Z)$ is also injective 
\end{theorem}
Although the proof is analogous to the proof given in \cite{TW}, we include the complete proof of 
injectivity here  so that we may observe that the proof works in this orbifold set-up, and 
under certain conditions, it holds with $\Z$-coefficients.

\proof For each connected component $F'$ of $F$, the group $\sfT_{F'}:=\sfT_{x,1}$, for some $x \in F'$, is the maximal global isotropy subtorus of $\sfT$ for $F'$ by  Lemma \ref{global isotropy}. Each weight of the action of $\sfT_{F'}$ on the negative normal bundle $E^-|_{F'}$ is non-trivial, which implies the condition $(\Q 2)$ in Proposition \ref{s.e.s}.  Therefore all lemmata and proposition in this section hold for our setup. 

Let $\xi$ is a generic element in $\ft$ and $c_1 < \cdots < c_n$ the critical values for $\mu^{\xi}$. Choose a small $\epsilon > 0$ such that $c_i$ is the only critical value in $[a_i,b_i]:=[c_i-\epsilon, c_i+\epsilon], \ \forall i$.  \\Let $M_{a_i}^-:=(\mu^{\xi})^{-1}(-\infty,a_i)$, $M_{b_i}^-:=(\mu^{\xi})^{-1}(-\infty,b_i)$, and $F_{c_i}:=F \cap (\mu^{\xi})^{-1}(c_i)$. We will prove that the restriction map $H_{\sfT}^*(M_{b_i}^-) \to H_{\sfT}^*(F \cap M_{b_i}^-)$ is injective for all $i$ and then the theorem follows from the case $i=n$, i.e. $M_{b_n}^-=M$.
We have the following commutative diagram with the exact horizontal rows:
{\tiny
\begin{eqnarray}\label{injectivity diagram}
\begin{array}{c}
\xymatrix{
0 \ar[r] & H^*_{\sfT}(M_{b_i}^-,M_{a_i}^-)\ar[d]_{\gamma_i} \ar[r] & H^*_{\sfT}(M_{b_i}^-) \ar[r] \ar[d]& H^*_{\sfT}(M_{a_i}^-) \ar[d]_{\beta_i}\ar[r] & 0\\
0\ar[r] & H^*_{\sfT}(F \cap M_{b_i}^-,F \cap M_{a_i}^-) \ar@{=}[d]\ar[r] & H^*_{\sfT}(F \cap M_{b_i}^-)\ar@{=}[d]\ar[r]&H^*_{\sfT}(F\cap M_{a_i}^-)\ar[r]\ar@{=}[d]& 0 \\
&H^*_{\sfT}(F_{c_i})&\bigoplus_{l=1}^{i} H^*_{\sfT}(F_{c_i})& \bigoplus_{l=1}^{i-1} H^*_{\sfT}(F_{c_i})&
}\end{array}
\end{eqnarray}
}
The top sequence is exact by Proposition~\ref{s.e.s}, and the second is clearly exact.  The vertical maps are restriction maps and we prove that they are injective by induction. The base case of the induction is trivial since $M_{a_1}^-$ is empty.  Now, the last column can be identified with  $H^*_{\sfT}(M_{b_{i-1}}^-) \to H^*_{\sfT}(F\cap M_{b_{i-1}}^-)$ by using  $H^*_{\sfT}(M_{b_{i-1}}^-) \cong_{homot} H^*_{\sfT}(M_{a_i}^-).$ The map  $\gamma_i$ is injective because it is the map in the Proposition~\ref{s.e.s},
\begin{eqnarray}\label{next diagram}
\begin{array}{c}
\xymatrix{
& H^i_{\sfT}(M_b^-,M_a^-) \ar[r]\ar[d]_{\cong_{MB}} \ar[ddr]_{\gamma}& H^i_{\sfT}(M_b^-) \ar[d] && \\
& H^i_{\sfT}(D_c,S_c)\ar[d]_{\cong_{Thom}}   & H^i_{\sfT}(D_c)\ar[d]^{\cong_{Homot}}  &&\\
& H_{\sfT}^{i-\lambda}(F_c) \ar[r]_{\cup \be_{\sfT}(E_c^-)}& H^i_{\sfT}(F_c) &
}
\end{array} .
\end{eqnarray}
The map $\beta_i$ in Diagram (\ref{injectivity diagram}) is injective by the induction hypothesis. 
Hence, by the Five Lemma, the middle map is injective.  

If $(\Z 3)$ is satisfied, then Proposition~\ref{s.e.s} holds with $\Z$ coefficients, and the above argument
carries through with $\Z$ coefficients. This completes the proof. \qed

%%%%%%%%%%%%%%%%%%%%%%%%%%%%%%%%%%
\begin{remark}\label{injectivity over Z}
We will show that the toric orbifolds with the maximal torus $\sfR$-action satisfies the condition $(\Z 3)$, and so
Theorem~\ref{thm:injectivity} holds over $\Z$.  Following Example~\ref{ex:torsion}, if we let $F$ be any manifold
with $2$-torsion in its integral cohomology, and let $S^1$ act on $F\times \C P^1$ by fixing $F$ pointwise and
rotating $\C P^1$ at double speed, then $(\Z 3)$ fails, as in Example~\ref{ex:torsion}, a certain Euler class will
be a zero-divisor and 
$$i^*:H_{S^1}^*(F\times \C P^1;\Z)\to H_{S^1}^*( (F\times \C P^1)^{S^1};\Z)$$ 
is not injective.  Indeed, if $\alpha$ is any $2$-torsion class in $H^*(F;\Z)$, then 
$$\alpha\otimes 1\in H^*(F;\Z)\otimes H_{S^1}^*(\C P^1;\Z)\cong H_{S^1}^*(F\times \C P^1;\Z)$$ 
is in the kernel of $i^*$.
%For the above theorem to hold over $\Z$, all lemmata and propositions in this section must hold for $\Z$ coefficients.  
%For that, we must assume that, for each connected component $F'$ of $F$, each weight of the action of $\sfT_{F'}$ 
%on the negative normal bundle $E^-|_{F'}$ is primitive.
\end{remark}
%%%%%%%%%%%%%%%%%%%%%%%%%%%%%%%%%%
%%%%%%%%%%%%%%%%%%%%%%%%%%%%%%%%%%
%%%%%%%%%%%%%%%%%%%%%%%%%%%%%%%%%%
%%%%%%%%%%%%%%%%%%%%%%%%%%%%%%%%%%
%%%%%%%%%%%%%%%%%%%%%%%%%%%%%%%%%%
%%%%%%%%%%%%%%%%%%%%%%%%%%%%%%%%%%
%%%%%%%%%%%%%%%%%%%%%%%%%%%%%%%%%%
%%%%%%%%%%%%%%%%%%%%%%%%%%%%%%%%%%
%%%%%%%%%%%%%%%%%%%%%%%%%%%%%%%%%%
%%%%%%%%%%%%%%%%%%%%%%%%%%%%%%%%%%
%%%%%%%%%%%%%%%%%%%%%%%%%%%%%%%%%%
%%%%%%%%%%%%%%%%%%%%%%%%%%%%%%%%%%
%%%%%%%%%%%%%%%%%%%%%%%%%%%%%%%%%%
%%%%%%%%%%%%%%%%%%%%%%%%%%%%%%%%%%
%%%%%%%%%%%%%%%%%%%%%%%%%%%%%%%%%%
%%%%%%%%%%%%%%%%%%%%%%%%%%%%%%%%%%
%%%%%%%%%%%%%%%%%%%%%%%%%%%%%%%%%%
%%%%%%%%%%%%%%%%%%%%%%%%%%%%%%%%%%
%%%%%%%%%%%%%%%%%%%%%%%%%%%%%%%%%%
%%%%%%%%%%%%%%%%%%%%%%%%%%%%%%%%%%
%%%%%%%%%%%%%%%%%%%%%%%%%%%%%%%%%%
%%%%%%%%%%%%%%%%%%%%%%%%%%%%%%%%%%
%%%%%%%%%%%%%%%%%%%%%%%%%%%%%%%%%%
\section{Generalizing Tolman and Weitsman's proof of the GKM theorem}
%%%%%%%%%%%%%%%%%%%%%%%%%%%%%%%%%%
We  need the following technical lemma, which generalizes Lemma 3.2. \cite{TW}.
%%%%%%%%%%%%%%%%%%%%%%%%%%%%%%%%%%
\begin{lemma}\label{lemma5.5}
Let $E$ be a $\sfT$-equivariant complex vector bundle over a manifold $F$ with an action of $\sfT$. Assume that $[F/\sfS ]$ is an orbifold and $\forall x\in F$, $\sfT=\sfT_x\cdot\sfS $.  Assume that $\sfT_{x,1}$ is independent of $x\in F$. Let $E:=\bigoplus_{\alpha}E_{\alpha}$ be the weight decomposition of the $\sfT_{x,1}$-action.  Then over $\Q$, for each $\eta \in H_{\sfT}^*(F)$, we have
\[
\mbox{If } \eta  \mbox{ is a multiple of } \be_{\sfT}(E_{\alpha}) \mbox{ for all }
\alpha,   \mbox{then } \eta \mbox{ is a multiple of } \cup_{\alpha} \be_{\sfT}(E_{\alpha}) = \be_{\sfT}(E).
\]
If we assume $\sfT_x$ is connected for all $x\in F$, this statement also holds over $\Z$.
\end{lemma}
\proof First we prove in the case when $[F/\sfS ]$ is $0$-dimensional. Over $\Q$, we have the sequence of isomorphisms,
\begin{eqnarray*}
H_{\sfT}^*(F) & = & H^*(E\sfT\times_{\sfT} F) \\
& = & H^*(B\sfT_x\times E(\sfS /\sfS _x)\times_{\sfS /\sfS _x} F) \\
& \cong & H^*(B\sfT_x) \\
& \cong & H^*(B\sfT_{x,1}) \\
& = & \Sym \left(\Hom(\sfT_{x,1}, S^1)\otimes \Q\right) \\
& = & \Q[\alpha_1,\cdots,\alpha_n].
\end{eqnarray*}
 The  equality in the second line holds because $\sfS /\sfS_x$ acts freely on $F$. Via this identification, $\be_{\sfT}(E_{\alpha}) = \alpha^{n_{\alpha}}$ where $n_{\alpha}$ is the rank of $E_{\alpha}$. Since $\alpha$'s are all distinct and non-zero, all $\be_{\sfT}(E_{\alpha})$'s are pairwise relatively prime. Hence $\eta$ must be a multiple of $\cup_{\alpha} \be_{\sfT}(E_{\alpha})$ since $\Q[\alpha_1,\cdots,\alpha_n]$ is a unique factorization domain over $\Q$.  Over $\Z$, the same argument works as long as we have $H_{\sfT}(F,\Z)\cong \Sym \left(\Hom(\sfT_{x,1}, S^1)\right) =\Z[\alpha_1,\cdots,\alpha_n].$ This happens precisely when $\sfT_x$ is connected.  

When $[F/\sfS]$ is not $0$-dimensional, we still have, over $\Q$,
\[ 
H_{\sfT}(F) = H(E\sfT\times_{\sfT} F) = H(B\sfT_{x,1} \times E(\sfS /\sfS_{x,1})\times_{\sfS /\sfS_{x,1}} F) \cong H_{\sfT_{x,1}}(pt)\otimes H(F/\sfS ).
\]
Define the $F$-degree by $H_{\sfT}(F)=\bigoplus_i H_{\sfT_{x,1}}(pt)\otimes H^i(F/\sfS ), \ \ \ a = \sum_i a_i$. Then $\be_{\sfT}(E_{\alpha})_0=\alpha^{n_{\alpha}}$ since the class is determined by pulling back $E_{\alpha}$ via $\{x\} \inc F$. The remainder of the proof is purely algebraic and so the argument in \cite[p.\ 8]{TW} can be followed verbatim.  

Over $\Z$, the above arguments works if  $H_{\sfT}(F) \cong H_{\sfT_{x,1}}(pt)\otimes H(F/\sfS )$. This happens exactly if $\sfT_x$ is connected for all $x \in F$. \qed
\\

Let $[M/\sfS ]$ be a compact, connected Hamiltonian $\sfR$-orbifold where $\sfS  \inc \sfT \twoheadrightarrow \sfR$ is the torus extension such that $\sfT$ acts on $M$ extending the $\sfS$-action. The corresponding maps of Lie algebras are $\fraks \inc \ft \twoheadrightarrow \frakr$. Let $i: F \inc M$ and $j:F \inc M_1$ be the inclusion maps. 
%%%%%%%%%%%%%%%%%%%%%%%%%%%%%%%%%%
\begin{remark}
Let $N^{\circ}$ be a connected component of $M_1^{\circ}$, and let  $N$  be its closure. Then $[N/\sfS ]$ is a Hamiltonian $\sfR$-orbifold with the moment map induced by the map $\mu|_N: N \to \ft^*$, the restriction of the original $\mu: M \to \ft^*$ (see Remark \ref{normal form for 1 skeleton}). If the injectivity theorem holds over $\Q$ (resp.\ over $\Z$) for $M$, then it also holds over $\Q$ (resp.\ over $\Z$) for $N$. This is because the conditions $(\Q 2)$ (resp.\ $(\Z 2)$) for $M$ imply the ones for $N$.
\end{remark}
%%%%%%%%%%%%%%%%%%%%%%%%%%%%%%%%%%

We now turn to the constraints on a class restricted to the fixed point set.  Using a component 
of the moment map to order the critical values, at the first fixed set where a class is non-zero, 
its restriction must be a multiple of the equivariant Euler class of the negative normal bundle.

\begin{prop}\label{5.2}
Choose a generic $\xi \in \ft$ and let $F_c$ be a connected component of the intersection $F \cap f^{-1}(c)$, where $f:=\mu^{\xi}$ and $c \in \R$ is a critical value for $\mu^{\xi}$. Let $(a,b)$ be an open interval containing the single critical value $c$. Let $M_{1,b}^-:=M_1\cap f^{-1}(-\infty,b)$ and $F_a^-:=F \cap f^{-1}(-\infty,a)$. Then if 
$
\eta \in H^*_{\sfT}(M_{1,b}^-,\Q)$ satisfies $\eta|_{F_a^-} =0 \in H^*_{\sfT}(F_a^-,\Q)$, we must have that $\eta|_{F_c}$ is a multiple of $\be_{\sfT}(E_c^-) \in H^*_{\sfT}(F_c,\Q)$,
where $E_c^-$ is the $\sfT$-equivariant negative normal bundle of $F_c$ in $M$. 

Moreover, the claim holds with integer coefficients if 
\begin{itemize}
\item[$(\Z 4)$] the isotropy group $\sfT_x$ is connected for all $x \in F$, and each weight of the $\sfT_{x}$-action 
on the normal bundle 
to $F$ at $x$ is primitive for all $x\in F$ 

\noindent (i.e.\ $(\Z 3)$ is satisfied).
\end{itemize}
\end{prop}

\begin{remark}
The hypothesis that $\sfT_x$ be connected is a natural one.  This hypothesis may be exploited to extend a variety of rational cohomology results  to integral cohomology in equivariant symplectic geometry.  This is discussed in work in progress by the first author and Tolman \cite{HT}.  This hypothesis also arises in the work of Franz and Puppe \cite[Theorem 1.1; and Examples 5.3 and 5.4]{FP}.
\end{remark}

\proof There is a connected component $N^{\circ}$ of $M_1^{\circ}$ such that its closure $N$ contains $F_c$, since if $F_c \cap N\not=\phi$, then $F_c \subset N$ by Lemma \ref{comp 1 sk}. If $\eta|_{F_a^-}=0$, then $\eta|_{N \cap F_a^-}=\eta|_{F \cap N_a^- }=0$. In the proof of Theorem~\ref{thm:injectivity}  for the $\sfR$-action on $[N/\sfS ]$, in the middle of the induction step we have the injection $H^*_{\sfT}(N_a^-) \inc H^*_{\sfT}(F \cap N_a^-).$ Hence $\eta|_{F \cap N_a^- }=0$ implies $\eta|_{N_a^-}=0$. Apply Proposition \ref{s.e.s} to the $\sfT$-action on $N$: we obtain a short exact sequence and a commutative diagram
\[
\xymatrix{
0 \ar[r] & H^i_{\sfT}(N_b^-,N_a^-) \ar[r]\ar@{=}[d] & H^i_{\sfT}(N_b^-) \ar[r]^{\beta} \ar[d]^{restriction}  & H^i_{\sfT}(N_a^-) \ar[r]& 0 \\
& H_{\sfT}^{i-\lambda}(F_c) \ar[r]_{\cup \be_{\sfT}(E_{N,c}^-)}& H^i_{\sfT}(F_c) &
}
\]
where $E_{N,c}^-$ is the $\sfT$-equivariant negative normal bundle of $F_c$ in $N$. The exactness and the commutativity of those diagram implies that any element in the kernel $\ker \beta$ is a multiple of the equivariant Euler class $\be_{\sfT}(E_{N,c}^-)$ when it is restricted to $F_c$. Thus by Lemma \ref{lemma5.5} above, we are done. \qed
%%%%%%%%%%%%%%%%%%%%%%%%%%%%%%%%%%
%%%%%%%%%%%%%%%%%%%%%%%%%%%%%%%%%%
%%%%%%%%%%%%%%%%%%%%%%%%%%%%%%%%%%
%%%%%%%%%%%%%%%%%%%%%%%%%%%%%%%%%%
%%%%%%%%%%%%%%%%%%%%%%%%%%%%%%%%%%
%%%%%%%%%%%%%%%%%%%%%%%%%%%%%%%%%%
%%%%%%%%%%%%%%%%%%%%%%%%%%%%%%%%%%
%%%%%%%%%%%%%%%%%%%%%%%%%%%%%%%%%%
\\

We have now established the preliminaries necessary to prove the orbifold version of Tolman 
and Weitsman's version of the GKM theorem.

\begin{theorem}
\label{orbifold GKM}
In the diagrams
\[
\begin{array}{c}
\xymatrix{
M  & M_1 \\
F\ar[u]^i \ar[ur]_j
}
\end{array}
\ \ \ \ \ \stackrel{\mbox{take } H_{\sfT}^*}{\Longrightarrow} 
\ \ \ \ \ 
\begin{array}{c}
\xymatrix{
H_{\sfT}(M)\ar[d]_{i^*}  & H_{\sfT}(M_1)\ar[dl]^{j^*} \\
H_{\sfT}(F)
}
\end{array},
\]
we have $\im i^*=\im j^*$ over $\Q$. In particular, $H^*_{\sfT}(M)\cong \im j^*$. If we assume that $\sfT_x$ is connected for all $x\in F$ and that the weights of the $\sfT_{x,1}$-action on the negative normal bundle (for a generic $\mu^{\xi}$) are primitive for all $x\in F$, the claim also holds over $\Z$.
\end{theorem}

\proof We have proved all the preliminary claims needed, so the proof goes through exactly as in \cite[p. 8--9]{TW}.  We proceed by induction on the index $i=1,\cdots,n$, where the critical values $c_i$ are ordered so that $c_1<\cdots < c_n$. Let $(a,b)$ be an open interval containing only $c_i$. Consider the following map of short exact sequences, using the notation from Proposition \ref{5.2}:
\[
\begin{array}{c}
\xymatrix{
0 \ar[r] & H_{\sfT}^*(M_b^-,M_a^-) \ar[d]\ar[r] & H_{\sfT}^*(M_b^-) \ar[d]\ar[r] & H_{\sfT}^*(M_a^-) \ar[r]\ar[d]& 0 \\
0 \ar[r] & \ker r^*|_{\im j_b^{-*}} \ar[r]& \im j_b^{-*} \ar[r]& \im j_a^{-*} \ar[r]& 0 \\
}\end{array},
\]
where $r: M_{1,a}^- \to M_{1,b}^-$ is the obvious inclusion. The third vertical map is surjective by the inductive assumption. The surjectivity of the first vertical map follows from Proposition 5.2. The surjectivity of the middle vertical map then follows by a diagram chase analogous to the one in the Five Lemma.
\qed
%%%%%%%%%%%%%%%%%%%%%%%%%%%%%%%%%%
%%%%%%%%%%%%%%%%%%%%%%%%%%%%%%%%%%
%%%%%%%%%%%%%%%%%%%%%%%%%%%%%%%%%%
%%%%%%%%%%%%%%%%%%%%%%%%%%%%%%%%%%
\section{GKM computations for toric orbifolds}\label{comb}
In this section, we compute the $\sfR$-equivariant cohomology of $[M/\sfS ]$ for compact symplectic toric orbifolds. We show that it is isomorphic, with $\Z$-coefficients, to the Stanley-Reisner ring of the corresponding moment polytope. This generalizes the case of smooth toric manifolds. This result should also follow from an argument using the moment angle complex, since the moment angle complex only depends on the combinatorial type of the polytope (see \cite{BP}, \cite{P}).  Nonetheless, we will now apply our techniques to obtain the following.

\begin{theorem}\label{stanley reisner for orbifold}
Let $\sfS $ be an $(m-n)$-dimensional subtorus of the $m$-dimensional torus $\sfT$ which acts on $\C^m$ coordinate-wise, and let $\mu_{\sfS}: \C^m \to \fraks^*$ be the induced moment map of the $\sfS$-action, so that $M:=\mu_{\sfS}^{-1}(\eta)$ is a compact manifold, for a regular value $\eta \in  \fraks^*$.  Let $\Delta$ be the moment polytope of the compact toric orbifold obtained as the symplectic reduction $\C^m/\!/ \sfS = [M/\sfS] $ at the regular value $\eta$. Then
\[
H^*_{\sfR}\left(\C^m/\!/ \sfS ,\Z\right) \cong \SR(\Delta)
\]
where $\SR(\Delta)$ is the Stanley-Reisner ring of the polytope $\Delta$.
\end{theorem} 

The (ordinary) Chow of an algebraic toric orbifold has been computed by Iwanari \cite{I} as
the quotient of the Stanley-Reisner ring by linear terms.
The ordinary integral cohomology $H^*(X;\Z)$ need not be the quotient of the Stanley-Reisner ring by linear terms.
As discussed in \cite{H} after the proof of Theorem~4.2, the integral cohomology of a direct product of two identical
weighted projective spaces has torsion in odd degrees, whereas the Stanley-Reisner ring can only contribute to 
even degrees.  This is explored further in \cite{LMM}.

%%%%%%%%%%%%%%%%%%%%%%%%%%%%%%%%%%
\subsection{Stanley Reisner ring and the direct sum decomposition}\label{SR}
Let $\Delta$ be a simple polytope of $n$-dimension. We let $K_{\Delta}$ denote the 
associated simplicial complex  whose vertices are the facets of $\Delta$, and a 
collection of vertices is a simplex in $K_{\Delta}$ if and only if  the corresponding 
collection of facets in $\Delta$ has non-empty intersection. The Stanley-Reisner 
ring $\SR(\Delta)$ of $\Delta$ is defined as the Stanley-Reisner ring 
$\SR(K_{\Delta})$  of the simplicial complex $K_{\Delta}$ associated to $\Delta$. 
Namely, 
\[
\SR(\Delta):=\frac{\Z[x_1,\cdots,x_m]}{\lan \prod_{i \in \sigma} x_i \ |\ \sigma \not\in K_{\Delta} \ran}
\]
where $m$ is the number of facets of $\Delta$. We define the ring
\begin{eqnarray*}
R & := & \bigoplus_{v \ vertex\ of\ \Delta}  \Z[x_{i_1},\cdots, x_{i_n}\ |\ v=H_{i_1}\cap \cdots \cap H_{i_n}] \\
& = & \bigoplus_{v \ vertex\ of\ \Delta}  \Z[x_{i_1},\cdots, x_{i_n}\ |\ H^*_v=\{i_1,\cdots,i_n\}], 
\end{eqnarray*}
where $\{H_1,\cdots, H_m\}$ is the set of facets of $\Delta$ and $H^*_v$ is the facet of $K_{\Delta}$ corresponding to the vertex $v$ of $\Delta$. We define a subring of $R$ by
\[
R_{\Delta} := \left\{ (p_{v_1},\cdots,p_{v_l}) \in R\ \left|\  \begin{array}{c}
p_v|_{x_j=0} = p_{v'}|_{x_i=0}\\ 
\forall \mbox{ edge } (v,v')\mbox{ in }\Delta, \\ 
\mbox{where }  j \in H^*_v \backslash H^*_{v'} \ \ i \in H^*_{v'} \backslash H^*_{v}
\end{array} \right.\right\}.
\]
This ring $R_{\Delta}$ is the algebra of continuous piece-wise polynomial functions on the fan canonically defined by the simplicial complex $K_{\Delta}$ and is well-known to be isomorphic to $\SR(\Delta)$ via
\[
x_i \mapsto (p_{v_1}^i,\cdots, p_{v_l}^i)\ \ ,  \ \mbox{ where } \ \ \ p_v^i=\left\{\begin{array}{cc}\!\!0 & \mbox{ if } i \not\in H_v^* \\\! x_i & \mbox{ if } i \in H_v^* \end{array}\right..
\]
See, for example, \cite[\S1.3]{B2}, \cite{BR} or \cite{Bl}. 
%%%%%%%%%%%%%%%%%%%%%%%%%%%%%%%%%%
\subsection{Injectivity over $\Z$ for toric orbifolds}
We recall the construction of symplectic toric orbifolds in \cite{LT}. Let $\Delta$ be a simple integral polytope in $\R^n$.  We may identify $\R^n \cong \frakr^*$. Let  $\calH=\{H_1,\cdots, H_m\}$ be the set of facets and $\rho_1,\cdots, \rho_m \in \Z^n$ the primitive inward normal vectors to facets and let $b_1,\cdots,b_m \in \Z_{>0}$ be positive integers that label the facets of $\Delta$ in the sense of \cite{LT}.  The polytope is given by
$$\Delta=\left\{ v \in \R^n\ \big|\ \lan b_i\rho_i, v\ran + \eta_i \geq 0, \ \ i=1,\cdots,m\right\}$$ 
for some $\eta:=(\eta_i) \in \R^n$. Let $B:=\left[\ b_1\rho_1\,|\,\cdots \,|\,b_m\rho_m\ \right] \in \mbox{Mat}_{n\times m}(\Z)$. Suppose that the transpose $^t B: \Z^n \to \Z^m$ has free cokernel so that we have a short exact sequence
\[
\xymatrix{ 0 \ar[r] & \Z^{m-n} \ar[r]^{A}& \Z^m \ar[r]^B & \Z^{n} \ar[r] & 0  }.
\]
The matrix $A \in \Mat_{m,m-n}(\Z)$ is given by choosing a basis of  $\ker(B)$. Now applying $\Hom(- ,S^1)$, we get an exact sequence of tori 
\[
\xymatrix{ 1 \ar[r] & \sfS^{(m-n)} \ar[r]^{\tilde{A}}& \sfT^{(m)} \ar[r]^{\tilde{B}} & \sfR^{(n)} \ar[r] & 1  },
\] 
and the corresponding exact sequence of Lie algebras 
\[
\xymatrix{ 0 \ar[r] & \fraks^{(m-n)} \ar[r]^{A}& \ft^{(m)} \ar[r]^B & \frakr^{(n)} \ar[r] & 0  }.
\]
Let $\overline\mu: \C^m \to  \fraks^*$ be the moment map for the action of $\sfS $ defined by the standard $\sfT$-action on $\C^m$ through the exact sequence above. This sends $(z_1,\cdots,z_m) \in \C^m$ to
\[
z =(z_1,\cdots, z_m) \mapsto\
^t\!A\cdot \left[\begin{array}{c}|z_1|^2 \\ \vdots \\ |z_m|^2\end{array}\right] = \left[\begin{array}{c}a_{11}|z_1|^2 + \cdots + a_{m1}|z_m|^2\\ \vdots \\ a_{1,m-n}|z_1|^2  + \cdots + a_{m,m-n}|z_m|^2\end{array}\right] ,
\]
where $^tA =\ ^t(a_{ij})_{1\leq i \leq m \atop{1\leq j\leq m-n }}.$ The orbifold $[M/\sfS ^{(m-n)}]$ corresponding to the labeled poyltope $(\Delta,b)$ is given by reduction at $\eta':=\ ^tA\cdot \eta \in \fraks^*$. Namely, 
\[
M=\mu^{-1}(\eta') = \left\{z \in \C^m\ |\ \ ^t\!A\cdot |z|^2 =\eta'\right\}.
\]
%%%%%%%%%%%%%%%%%%%%%%%%%%%%%%%%%%
\begin{lemma}[\cite{LT}, proof of Theorem 8.1]
Let $v=H_{i_1} \cap \cdots \cap H_{i_n} \in \frakr^* \cong \R^n$ be a vertex of $\Delta$. Then the corresponding fixed orbifold point in $[M/\sfS ]$ is given by $[F_v/\sfS ]$ where 
\[
F_v=\{(z_1,\cdots, z_m) \in \C^m \ \big|\ |z_i|^2=( ^t\!B\cdot v + \eta)_i, \ \ i=1,\cdots,m \},
\]
and $( ^t\!B\cdot v + \eta)_i=0$ if and only if $i=i_1,\cdots, i_n$. For each $x \in F_v$, the isotropy group $\sfT_x$ is a subtorus of $\sfT$ 
\[
\sfT_x=\left\{(t_1,\cdots,t_m) \in \sfT \ \big|\ t_i=1\ \ \forall i \in [m]\backslash \{i_1,\cdots,i_n\}\right\}.
\]
In particular, the isotropy $\sfT_{F_v}$ of $F_v$ equals $\sfT_x$ for every $x \in F_v$.
\end{lemma}

\proof This follows from the proof of Theorem 8.1 in \cite{LT}. Since a vertex $v \in \Delta \subset \R^n$ is given by $n$ equations $\lan b_{i_1}\rho_{i_1},v \ran  +\eta_{i_1}=0,\  \cdots , \  \lan b_{i_n}\rho_{i_n},v \ran +\eta_{i_n}=0$, we have  $(^t\!B\cdot v + \eta)_i=0$ if and only if $i=i_1,\cdots, i_n$. \qed
\

\

%%%%%%%%%%%%%%%%%%%%%%%%%%%%%%%%%%
The local normal form around a point $p_v$ in $F_v$ is given by
\[
U_{p_v} \cong \sfT \times_{\sfT_{F_v}} W, \ \mbox{ where } W \cong \C\cdot \frac{\partial}{\partial z_{i_1}} \oplus \cdots \oplus \C\cdot \frac{\partial}{\partial z_{i_n}}
\]
and the weight of $\C\cdot \frac{\partial}{\partial z_{i_k}}, k=1,\cdots,n$ is $\lambda_{i_k} \in \Hom(\sfT_{F_v},S^1)$ defined by $\lambda_{i_k}(t)=t_{i_k}$ for all $t \in \sfT_{F_v}$. Therefore, the action satisfies hypothesis $(\Z 2)$.
Thus Theorem \ref{orbifold GKM} holds over $\Z$.  
\begin{prop}
The restriction map $H_{\sfT}^*(M,\Z) \to H^*_{\sfT}(F,\Z)$ is injective where $F$ is the union of $F_v$'s for all vertices $v \in \Delta$ and its image coincides with the image of $H_{\sfT}^*(M_1, \Z)\to  H^*_{\sfT}(F,\Z)$.
\end{prop}

%%%%%%%%%%%%%%%%%%%%%%%%%%%%%%%%%%
\subsection{The 1-skeleton and GKM computations}
Recall that $H_v^*$ is the facet of the associated simplicial complex $K_{\Delta}$ corresponding to a vertex $v$ of $\Delta$.  For each edge $(v,u)$ of $\Delta$,  we have $|H_v^*\cap H_u^*|=n-1$. Letting $\{a\} =H_v^*\backslash H_u^*$ and $\{b\} = H_u ^*\backslash H_v^*$, the corresponding component of the 1-skeleton is
\begin{eqnarray*}
N_{v,u} 
&=&\left\{(z_1,\cdots,z_m) \in \C^m\ \big|\ |z_k|^2 = s\overline v_k +(1-s) \overline u_k, \ s\in[0,1]  \right\}, 
\end{eqnarray*}
where $\overline v = \ ^tB\cdot v + \eta$. Therefore, $(z_1,\cdots,z_m) \in N_{v,u}$ is given by
\begin{eqnarray*}
|z_k|^2 &=& 0  \ \ \ \ \ \ \forall k \in H_v^*\cap H_u^*,\\
|z_b|^2 &=& s\overline v_b \not=0,\\
|z_a|^2&=& (1-s)\overline u_a\not=0,\\
|z_l|^2 &=& s\overline v_l +(1-s)\overline u_l\not=0 \ \ \ \ \ \forall l \not\in H_v^*\cup H_u^*.
\end{eqnarray*}
By getting rid of the parameter $s$, we can also write $N_{u,v}$  as
\begin{eqnarray*}
|z_k|^2 &=& 0  \ \ \ \ \ \ \forall k \in H_v^*\cap H_u^*,\\
\overline u_a |z_b|^2 + \overline v_b |z_a|^2  &=& 2 \overline u_a \overline v_b \not=0,\\
|z_l|^2 &=& \frac{|z_b|^2}{\overline v_b}\overline v_l +(1-\frac{|z_b|^2}{\overline v_b})\overline u_l\not=0 \ \ \ \ \ \forall l \not\in H_v^*\cup H_u^*.
\end{eqnarray*}
The pair $(\C^m, N_{v,u})$ is $\sfT$-equivariantly homotopic to the pair $$(\C^m, \{0\}^{n-1}\times S^3 \times (S^1)^{m-n-1}).$$  We have the following short exact sequence (cf. Theorem 4.2 \cite{H})
\[
0 \to H^i_{(S^1)^k}(\C^k, S^{2k-1}) \to H^i_{(S^1)^k}(\C^k) \to H^i_{(S^1)^k}(S^{2k-1}) \to 0.
\]
By applying the K\"{u}nneth formula, we obtain a surjection 
$H^*_{\sfT}(\C^m) \surj H^*_{\sfT}(N_{v,u})$. 
\begin{lemma}\label{lemma holm}
The inclusion $\pi': N_{v,u} \inc \C^m$ induces a surjective map 
\[
\pi'^*: H^*_{\sfT}(\C^m) \to H^*_{\sfT}(N_{v,u}).
\] 
Furthermore, since 
$\C^m$ equivariantly retracts to $\{0\}$, by the commutativity of the diagram
\[
\begin{array}{c}
\xymatrix{
&&E\sfT \times_{\sfT} \C^m \ar@{=}[d]^{homot} &&\\
E\sfT \times_{\sfT} N_{v,u}\ar[rru]^{\pi'} \ar[rr]_{\ \ \ \ \ \ \ \ \pi}&& B\sfT
}\end{array},
\]
$\pi^*$ is surjective. 
\end{lemma}

Now consider the diagram
\[
\begin{array}{c}
\xymatrix{
E\sfT\times_{\sfT}(N_{v,u}) \ar[r]^{\pi} & B\sfT \\ 
E\sfT\times_{\sfT} F_v \sqcup E\sfT\times_{\sfT} F_{u} \ar[u]^i\ar[ru]_{\rho} & }
\end{array}.
\]
Taking cohomology with $\Z$ coefficients, we obtain
\[
\begin{array}{c}
\xymatrix{
H_{\sfT}(N_{v,u}) \ar[d]_{i^*} & \Z[x_1,\cdots, x_m] \ar[l]_{\pi^*} \ar[ld]^{\rho^*}\\
\Z[x_{i_1},\cdots,x_{i_{n-1}},x_a] \oplus \Z[x_{i_1},\cdots,x_{i_{n-1}},x_b] &
}\end{array},
\]
where $H_{v}^*=\{i_1,\cdots,i_{n-1},a\}$ and $H_{u}^*=\{i_1,\cdots,i_{n-1},b\}$. The image of $i^*$ consists $(p_v,p_u)$ such that $p_v|_{x_a=0}=p_u|_{x_b=0}$. Thus the image of $H^*_{\sfT}(M_1)$ in $H_{\sfT}(F)$ is
\[
\left\{ (p_{v_1},\cdots,p_{v_l}) \in R\ \left|\ \begin{array}{c}
 p_v|_{x_j=0} = p_{v'}|_{x_i=0}\\ 
 \forall \mbox{ edge } (v,v')\mbox{ in }\Delta\\
 \mbox{where }  \{j \}= H^*_v \backslash H^*_{v'} \ \ \{i \}= H^*_{v'} \backslash H^*_{v}  \end{array}
 \right. \right\},
\]
which coincides with $R_{\Delta}$ defined in Section \ref{SR}.  Thus we may conclude that the $\sfR$-equivariant cohomology of the toric orbifold $[M/\sfS ]$ is isomorphic to the Stanely-Reisner ring of $\Delta$.  This completes the proof of Theorem~\ref{stanley reisner for orbifold}.
%%%%%%%%%%%%%%%%%%%%%%%%%%%%%%%%%%
%%%%%%%%%%%%%%%%%%%%%%%%%%%%%%%%%%
%%%%%%%%%%%%%%%%%%%%%%%%%%%%%%%%%%
%%%%%%%%%%%%%%%%%%%%%%%%%%%%%%%%%%
%%%%%%%%%%%%%%%%%%%%%%%%%%%%%%%%%%
%%%%%%%%%%%%%%%%%%%%%%%%%%%%%%%%%%
%%%%%%%%%%%%%%%%%%%%%%%%%%%%%%%%%%
%%%%%%%%%%%%%%%%%%%%%%%%%%%%%%%%%%
%%%%%%%%%%%%%%%%%%%%%%%%%%%%%%%%%%
%%%%%%%%%%%%%%%%%%%%%%%%%%%%%%%%%%
%%%%%%%%%%%%%%%%%%%%%%%%%%%%%%%%%%
%%%%%%%%%%%%%%%%%%%%%%%%%%%%%%%%%%
%%%%%%%%%%%%%%%%%%%%%%%%%%%%%%%%%%
%%%%%%%%%%%%%%%%%%%%%%%%%%%%%%%%%%
%%%%%%%%%%%%%%%%%%%%%%%%%%%%%%%%%%
%%%%%%%%%%%%%%%%%%%%%%%%%%%%%%%%%%
%%%%%%%%%%%%%%%%%%%%%%%%%%%%%%%%%%
%%%%%%%%%%%%%%%%%%%%%%%%%%%%%%%%%%
%%%%%%%%%%%%%%%%%%%%%%%%%%%%%%%%%%
%%%%%%%%%%%%%%%%%%%%%%%%%%%%%%%%%%
\section{$\sfR$-equivariant Chen-Ruan orbifold cohomology of $[M/\sfS ]$}\label{sec:eqCR}
%%%%%%%%%%%%%%%%%%%%%%%%%%%%%%%%%%

We now turn to $\sfR$-equivariant Chen-Ruan theory.  When $\sfR$ is the trivial group, our definitions agree with the usual non-equivariant ones \cite{CR1}.  After defining the $\sfR$-equivariant Chen-Ruan orbifold cohomology ring of the $\sfR$-orbifold $[M/\sfS ]$, we survey the literature on this topic which motivates our definitions and results. The inertia manifold for the locally free $\sfS $-action on $M$ is defined by $\cI_{\sfS } M := \bigsqcup_{g \in \sfS } M^{g}$, where $M^g$ is the set of fixed points by the subgroup $\lan g \ran$ generated by $g$. This disjoint union is a finite union, since $M$ is compact and the $\sfS $-action is locally free. There is also an induced $\sfT$-action on $M$, which allows us to define the Hamiltonian $\sfR$ action on the orbifold $$\cI[M/\sfS ]:= \bigsqcup_{g \in \sfS } [M^{g}/\sfS ],$$ called the {\bf inertia orbifold}. We may then define the {\bf $\sfR$-equivariant Chen-Ruan cohomology} of $[M/\sfS ]$, as a vector space, to be 
\[
H_{orb,\sfR}([M/\sfS ]):=H_{\sfT}(\cI_{\sfS }M)=\bigoplus_{g \in \sfS } H_{\sfT}(M^g) = \bigoplus_{g \in \sfS } H_{\sfR}([M^g/\sfS ]).
\] 
Let $M^{g,h}:=M^g\cap M^h$. The normal bundle $N_{M^g\subset M}$ of $M^g$ in $M$ is then a $\sfT$-equivariant complex vector bundle, with weight decomposition 
\[
N_{M^g\subset M} = \bigoplus_{\lambda \in \Hom(\lan g\ran,S^1)} W_{\lambda}.
\] 
Define an element $\calS_g$ of the $\sfT$-equivariant (topological) $K$-theory $K_{\sfT}(M^g)\otimes\Q$ of $M^{g,h}$ over $\Q$ by
\[
\calS_g = \bigoplus_{\lambda \in \Hom(\lan g\ran,S^1)} a_{\lambda}(g) W_{\lambda}
\]
where $a_{\lambda}(g) \in [0,1)$ is the {\bf age}, defined by $\lambda(g)=e^{2\pi i a_{\lambda}(g)}$. Following \cite{EJK, JKK}, define the {\bf equivariant virtual bundle} $\calR_{M}(g,h)$ as an element of $K_{\sfT}(M^{g,h})\otimes \Q$
\[
\calR_{M}(g,h):=\ominus N_{M^{g,h}\subset M} \oplus \cS_g|_{M^{g,h}} \oplus \cS_h|_{M^{g,h}} \oplus \cS_{(gh)^{-1}}|_{M^{g,h}}.
\]
 Since $H:=\lan g,h\ran$ acts on each tangent space $T_xM$ and $T_xM/T_xM^{g,h}$ for $x \in M^{g,h}$, we have the decomposition 
$$N_{M^{g,h}\subset M} = \bigoplus_{\lambda \in \Hom(H,S^1)} W_{\lambda}.$$ Since the $\sfT$-action commutes with the $H$-action, this decomposition is $\sfT$-stable; that is, each $W_{\lambda}$ is a $\sfT$-equivariant complex vector bundle. Then we may show that 
\[
\calR_M(g,h)=\bigoplus_{a_{\lambda}(g)+a_{\lambda}(h)+a_{\lambda}((gh)^{-1})=2, \atop{\lambda\not=0}} W_{\lambda}.
\]
Thus $\calR_{M}(g,h)$ is actually represented by a $\sfT$-equivariant complex vector bundle. This is the version of the obstruction bundle introduced in \cite{BCS}. Since $\calR(g,h)$ is a $\sfT$-equivariant complex vector bundle on $M^{g,h}$, we take the $\sfT$-equivariant Euler class to define
$$c_{M}(g,h):=\be_{\sfT}\left(\calR_{M}(g,h)\right) \in H_{\sfT}(M^{g,h}),
$$ 
called the {\bf virtual class}. We define the {\bf $\sfR$-equivariant orbifold product} on $H_{orb,\sfR}([M/\sfS ])$ by the usual pull-cup-push formula. Namely for $\eta \in H_{\sfT}(M^g)$ and $\xi \in H_{\sfT}(M^h)$, 
\begin{eqnarray}
\eta \odot \xi := e_* \left(e_1^*\eta \cup e_2^*\xi \cup c_{M}(g,h)\right),
\end{eqnarray}
where $e_1$, $e_2$, and $e$ are the obvious inclusions of $M^{g,h}$ into $M^g$, $M^h$, and $M^{gh}$ respectively. For the definition of the equivariant pushforward $e_*$, see for example \cite{AB} Section 2.  The associativity of this product follows immediately from the proof of the corresponding associativity in non-equivariant case in \cite{BCS, GHK, JKK}. For $\eta \in H^{|\eta|}(M^g)$, the rational grading is assigned by $\deg_{\Q} \eta = |\eta| + 2\cdot \mbox{age}(g)$,  where $\mbox{age}(g):=\mbox{rank} \cS_g$. It follows immediately that the product is rationally graded.  If one of $g$, $h$, or $gh$ is the identity, then the obstruction bundle has rank $0$ and so it is easy to see that $H^*_{\sfT}(M)$ sits in $H^*_{orb,\sfR}([M/\sfS ])$ as a subalgebra. In particular, $H^*_{orb,\sfR}([M/\sfS ])$ is a $H^*_{\sfT}(pt)$-algebra.
%%%%%%%%%%%%%%%%%%%%%%%%%%%%%%%%%%
\begin{theorem}\label{associativity}
$H^*_{orb,\sfR}([M/\sfS ])$ is a rationally graded, associative, $H^*_{\sfT}(pt)$-algebra.
\end{theorem}

The usual Chen-Ruan orbifold cohomology groups of an algebraic orbifold $\calX$ are defined 
as a vector space by $H^*(\calI\calX)$, where $\calI\calX$ is the inertia orbifold. The orbifold
product is the usual cup product which is then deformed by the Euler class of 
the obstruction bundle
for the corresponding Gromov-Witten theory \cite[Section 6]{AGV}. The obstruction bundle for
algebraic toric orbifolds has been computed by \cite{BCS} and adopted for symplectic
orbifolds by \cite{GHK} following the original definitions in \cite{CR1}. The most recent formula
for algebraic orbifolds that are global quotients by algebraic groups can be found in \cite{EJK}. 

If $\calX$ is an algebraic $\sfG$-orbifold, there is an induced action on the obstruction bundle defined in \cite{AGV}, and the $\sfG$-equivariant Chen-Ruan cohomology ring can be defined as the equivariant cohomology groups of the inertia orbifold together with the orbifold cup product deformed by the equivariant Euler class of the obstruction bundle \cite[Section 2.2.1]{J}. On the other hand, the obstruction bundle defined in \cite{CR1} for a symplectic $\sfG$-orbifold is also naturally $\sfG$-equivariant. In the case of the symplectic $\sfR$-orbifolds considered in this paper, the formula in \cite{GHK} derived from the definition in \cite{CR1} is $\sfR$-equivariantly valid. The argument in \cite[Appendix A]{GH} can be made equivariant.  It is also possible to verify that the computation done in \cite{BCS} is valid $\sfR$-equivariantly.

%%%%%%%%%%%%%%%%%%%%%%%%%%%%%%%%%%
%%%%%%%%%%%%%%%%%%%%%%%%%%%%%%%%%%
\section{Injectivity theorem for equivariant Chen-Ruan orbifold cohomology}\label{star product}
%%%%%%%%%%%%%%%%%%%%%%%%%%%%%%%%%%
The ring $H^*_{orb,\sfR}([M/\sfS ])$ is not functorial: a map between spaces may not induce a map on Chen-Rual orbifold 
cohomology rings.  In particular, the inclusion of the fixed points $[F/\sfS]\into [M/\sfS]$ does not induce a map
in Chen-Ruan orbifold cohomology. Thus we introduce a new ring $\calNH^*_{\sfR}(\nu[F/\sfS ])$.  Recall from 
Section~\ref{se:fixed} that 
$$F:= \{ x\in M\ |\ \sfT\cdot x = \sfS\cdot x\}$$ 
is a submanifold of $M$, and the suborbifold
$[M/\sfS]^{\sfR}$ of $\sfR$-fixed orbifold points of $[M/\sfS]$ is exactly $[F/\sfS]$.  The ring $\calNH^*_{\sfR}(\nu[F/\sfS ])$
will be defined for the normal bundle of $[F/\sfS$ in $[M/\sfS]$.  This new ring  is defined only using the fixed points and 
isotropy data at the fixed points.
%%%%%%%%%%%%%%%%%%%%%%%%%%%%%%%%%%
\begin{defn} As a vector space, we define $$\calNH_{\sfR}(\nu[F/\sfS ]):=\bigoplus_{g \in \sfS } H_{\sfT}(F^g).$$ The rational grading is defined with an age shift, exactly as in the previous section. We have the natural restriction map from $H_{orb,\sfR}([M/\sfS ])$ to $\calNH_{\sfR}(\nu[F/\sfS ])$ and in order to make it a ring homomorphism, the product must be defined appropriately in $\calNH_{\sfR}(\nu[F/\sfS ])$ using a push-cup-pull formula. Let $\calE(g,h)$ be the excess intersection bundle $\calE(g,h) = N_{M^{g,h}\subset M^{gh}}|_{F^{g,h}} \ominus N_{F^{g,h}\subset F^{gh}}$ for the diagram
\[
\begin{array}{c}
\xymatrix{
M^{g,h} \ar[r] & M^{gh} \ar@{}[ld]|{\square}\\
F^{g,h} \ar[r] \ar[u]& F^{gh}\ar[u]
}\end{array}.
\]
Define
\[
\calR'_F(g,h):=\calR_M(g,h)|_{F^{g,h}} \oplus \calE(g,h), \ \ \ c'_F(g,h):=\be_{\sfT}(\calR'_F(g,h)).
\]
The product $\star$ on $\calNH_{\sfR}(\nu[F/\sfS ])$ is defined for $a \in H_{\sfT}(F^g)$ and $b\in H_{\sfT}(F^h)$ by  $a\star b := f_{*}\left( f_1^*a \cup f_2^*b \cup c'_F(g,h)\right)$ where $f_1,f_2, f$ are the obvious inclusions of $F^{g,h}$ into $F^g, F^h$ and $F^{gh}$.
\end{defn}
%%%%%%%%%%%%%%%%%%%%%%%%%%%%%%%%%%
\begin{theorem}\label{inertial associativity}
$(\calNH^*_{\sfR}(\nu[F/\sfS ]),\star)$ is an associative graded ring.
\end{theorem}
\proof Let $g,h,m \in \sfS $. Denote all relevant inclusions by 
\[
\xymatrix{
F^g  &  F^{g,h}  \ar[r]_f \ar[l]_{f_1} \ar[ld]_{f_2} \ar@{}[rd]|{\square} &  F^{gh}&\\
F^h  & F^{g,h,m} \ar[u]_{\phi} \ar[r]^{\psi}          & F^{gh,m} \ar[lld]^{f_3} \ar[u]_{f_4} \ar[r]_l & F^{ghm}\\
F^m & &&
}
\]
and
\[
\xymatrix{
F^g  &   &&\\
F^h  & F^{g,h,m} \ar[lu]_{\overline f_1}\ar[l]_{\overline f_2} \ar[r]^{\psi}    \ar[ld]_{\overline f_3}      & F^{gh,m} \ar[lld]^{f_3}  \ar[r]_l & F^{ghm}\\
F^m & &&.
}
\]
Let us calculate
\begin{eqnarray*}
(a\star b) \star c &=& l_*\left(f_4^*f_{*}\left( f_1^*a \cup f_2^*b \cup c'_F(g,h)\right) \cup f_3^*c \cup c'_F(gh,m)\right)\\
&&\ \ \ \ \ \ \ \mbox{ by definition; } \\
&=& l_*\left(\psi_*\left(\phi^{*}\left( f_1^*a \cup f_2^*b \cup c'_F(g,h)\right)\cup \epsilon\right) \cup f_3^*c \cup c'_F(gh,m)\right)\\
&& \ \ \ \ \ \ \ \mbox{ by the excess intersection formula; }\\
&=&l_*\left(\psi_*\left(\phi^{*} f_1^*a \cup \phi^{*}f_2^*b \cup \phi^{*}c'_F(g,h) \cup  \epsilon\right) \cup f_3^*c \cup c'_F(gh,m)\right)\\
&&\ \ \ \ \ \ \  \mbox{ because pull-back commutes with cup product; }\\
&=&l_*\left(\psi_*\left(\overline f_1^*a \cup \overline f_2^*b \cup \phi^{*}c'_F(g,h) \cup  \epsilon\right) \cup f_3^*c \cup c'_F(gh,m)\right)\\
&&\ \ \ \ \ \ \  \mbox{ because $\overline f_1 = f_1\circ \phi$ and $\overline f_2 = f_2\circ \phi$; } \\
%\end{eqnarray*}
%\begin{eqnarray*}
&=&l_*\psi_*\left(\overline f_1^*a \cup \overline f_2^*b \cup \phi^{*}c'_F(g,h) \cup  \epsilon \cup \psi^*f_3^*c \cup \psi^*c'_F(gh,m)\right)\\
&& \ \ \ \ \ \ \  \mbox{ by the projection formula; }\\
&=&l_*\psi_*\left(\overline f_1^*a \cup \overline f_2^*b \cup \overline f_3^*c \cup \phi^{*}c'_F(g,h) \cup  \epsilon  \cup \psi^*c'_F(gh,m)\right).
\end{eqnarray*}
In the last line, we denote $\epsilon:=\be_{\sfT}(E)$ where the bundle 
$$E:=N_{F^{g,h}\subset F^{gh}}|_{F^{g,h,m}}  \ominus N_{F^{g,h,m}\subset F^{gh,m}} $$ is the  excess intersection bundle corresponding to the square in above diagram. Now $\phi^{*}c'_F(g,h) \cup  \epsilon  \cup \psi^*c'_F(gh,m)$ is the $\sfT$-equivariant Euler class of 
\begin{eqnarray*}
&& \phi^{*}\left(\calR_M(g,h)|_{F^{g,h}} \oplus \calE(g,h)\right) \oplus E  \oplus \psi^*\left(\calR_M(gh,m)|_{F^{gh,m}} \oplus \calE(gh,m)\right) \\
%&=& \ominus N_{M^{g,h} \subset M} \oplus \calS_g \oplus \calS_h \oplus \calS_{(gh)^{-1}}  \oplus N_{M^{g,h}\subset M^{gh}} \ominus N_{F^{g,h}\subset F^{gh}} \oplus N_{F^{g,h}\subset F^{gh}}  \ominus N_{F^{g,h,m}\subset F^{gh,m}}  \\
%&&\ \ \ \ \  \ominus N_{M^{gh,m} \subset M} \oplus \calS_{gh} \oplus \calS_m \oplus\calS_{(ghm)^{-1}} \oplus  N_{M^{gh,m}\subset M^{ghm}} \ominus N_{F^{gh,m}\subset F^{ghm}}\\
%&=& \ominus N_{M^{g,h} \subset M} \oplus N_{M^{g,h}\subset M^{gh}} \ominus N_{F^{g,h}\subset F^{gh}} \oplus N_{F^{g,h}\subset F^{gh}}  \ominus N_{F^{g,h,m}\subset F^{gh,m}}  \ominus N_{M^{gh,m} \subset M} \\
%&& \ \ \ \ \ \ \ \ \ \ \ N_{M^{gh,m}\subset M^{ghm}} \ominus N_{F^{gh,m}\subset F^{ghm}} \oplus \calS_g \oplus \calS_h \oplus \calS_{(gh)^{-1}} \oplus \calS_{gh} \oplus \calS_m \oplus\calS_{(ghm)^{-1}} \\
&=&  \ominus N_{F^{g,h,m}\subset F^{ghm}}  \ominus N_{M^{ghm} \subset M}  \oplus\calS_g \oplus \calS_h \oplus \calS_m \oplus\calS_{(ghm)^{-1}} ,
\end{eqnarray*}
where we omit $|_{F^{g,h,m}}$ everywhere, and the only non-obvious cancellation  uses 
$$\calS_g \oplus \calS_{g^{-1}} = TM \ominus TM^g.$$ 
The final form is symmetric in $(g,h,m)$, establishing the associativity of $\star$. \qed

\vskip 0.1in
%%%%%%%%%%%%%%%%%%%%%%%%%%%%%%%%%%
The following is an immediate generalization of the product to an $n$-fold product.
\begin{cor}
For $a_i \in H_{\sfT}(F^{g_i}), \ i=1,\cdots,n$, we may define an $n$-fold product by
\[
a_1\star a_2 \star \cdots \star a_n = \bbf_*\left(\overline f_1^*a_1 \cup \overline f_2^* a_2 \cup \cdots \cup \overline f_n^* a_n \cup \be_{\sfT}\left(\calR'_F(g_1,\cdots,g_n)\right)\right)
\]
where $\overline f_i:  F^{g_1,\cdots,g_n} \to F^{g_i}$ and $\bbf: F^{g_1,\cdots,g_n} \to F^{\prod g_i}$ are obvious inclusions and the obstruction bundle is given by $$\calR'_F(g_1,\cdots,g_n):=\ominus N_{F^{g_1,\cdots,g_n} \subset F^{\prod g_i}} \ominus N_{M^{\prod g_i}\subset M}\oplus \bigoplus_{i=1}^n \calS_{g_i} \oplus \cS_{\left(\prod g_i\right)^{-1}}.$$
\end{cor}
%%%%%%%%%%%%%%%%%%%%%%%%%%%%%%%%%%
%%%%%%%%%%%%%%%%%%%%%%%%%%%%%%%%%%
%%%%%%%%%%%%%%%%%%%%%%%%%%%%%%%%%%
%%%%%%%%%%%%%%%%%%%%%%%%%%%%%%%%%%
The inclusions $i:F^g \to M^g$ for all $g \in \sfS $ altogether induce a rationally graded linear map $\mathcal{I}:H_{orb,\sfR}([M/\sfS ]) \to \calNH_{\sfR}(\nu[F/\sfS ])$.  Consider the diagram of obvious inclusions:
\[ 
\xymatrix{
M^g\times M^h 					& M^{g,h} \ar[r]^e \ar[l]_{\ \ \ \ \ \ \Delta_M} \ar@{}[rd]|{\square}& M^{gh}  \\
F^g\times F^h \ar[u]^{(i_1,i_2)}	& F^{g,h} \ar[r]^f \ar[l]_{\ \ \ \ \ \ \Delta_F}  \ar[u]_{j} & F^{gh} \ar[u]_i
}
\] 
The map $\mathcal{I}$ is a graded ring homomorphism if  
\begin{eqnarray*}
i^*(e_*(\Delta_M^*(\eta\otimes\xi)\cup c_{\sfT}(g,h))) & = & f_*(\Delta_F^*(i_1,i_2)^*(\eta\otimes\xi)\cup c'_{\sfT}(g,h))\\
&  = &  f_*(j^*\Delta_M^*(\eta\otimes\xi)\cup c'_{\sfT}(g,h)).
\end{eqnarray*}
Since $c'_{\sfT}(g,h)=j^*(c_{\sfT}(g,h)) \cup \be_{\sfT}(\calE(g,h))$, the equality follows exactly from the excess intersection formula. Thus, combined with the injectivity theorem, we obtain
%%%%%%%%%%%%%%%%%%%%%%%%%%%%%%%%%%
\begin{theorem}\label{map of rings}
The natural map $\mathcal{I}:(H_{orb,\sfR}([M/\sfS ]),\odot) \to (\calNH_{\sfR}(\nu[F/\sfS ]),\star)$ is a graded ring homomorphism. If the $\sfT$-action on $M$ satisfies the condision $(\Q 2)$ (resp.\ $(\Z 2)$), then this homomorphism is injective over $\Q$ (resp.\ over $\Z$).
\end{theorem}
%%%%%%%%%%%%%%%%%%%%%%%%%%%%%%%%%%
%%%%%%%%%%%%%%%%%%%%%%%%%%%%%%%%%%
%%%%%%%%%%%%%%%%%%%%%%%%%%%%%%%%%%
%%%%%%%%%%%%%%%%%%%%%%%%%%%%%%%%%%
%%%%%%%%%%%%%%%%%%%%%%%%%%%%%%%%%%
\section{Examples: compact symplectic toric orbifolds}
%%%%%%%%%%%%%%%%%%%%%%%%%%%%%%%%%%
%%%%%%%%%%%%%%%%%%%%%%%%%%%%%%%%%%
%%%%%%%%%%%%%%%%%%%%%%%%%%%%%%%%%%
\subsection{Pullback and pushforward maps for inclusions of polytopes}
In this section, we collect the notions of pullback and pushforward maps of Stanley-Reisner rings. Let $\Delta$ be a simple polytope with $m$ facets $H_1,\cdots, H_m$. For $\tau \in K_\Delta$, let $G=\cap_{i \in \tau} H_i$ be an $(n-r)$-dimensional face of $\Delta$. Then $G$ is also a simple polytope and the corresponding simplicial complex $K_G$ is isomorphic to the link $K_{\Delta,\tau}$ of $\tau$ in $K_{\Delta}$, namely,  $K_{\Delta,\tau}:=\{\sigma \subset [m]\backslash \tau \ |\ \sigma \sqcup \tau \in K_{\Delta}\}$.
Let $K_{\Delta,\tau}*\tau$ be the joint of $K_{\Delta,\tau}$ with the simplex $\tau$, namely $K_{\Delta,\tau}*\tau:=\{\sigma_1\sqcup \sigma_2 \ \ \sigma_1 \in K_{\Delta,\tau} \mbox{ and } \sigma_2 \subset \tau\}$. Then
\[
\widetilde{\SR}(G):=\SR(K_{\Delta,\tau}*\tau) \cong \frac{\Z[x_1,\cdots,x_m]}{\lan x^{\sigma} \ |\ \sigma \subset [m]\backslash \tau, \ \sigma \cup \tau \in K_{\Delta}\ran}\cong \SR(G)\otimes \Z[x_i, i\in\tau],
\]
where $x^{\sigma}:=\prod_{i\in\sigma}x_i$. If $G'=\cap_{i\in \tau'} H_i$ be a non-empty face contained in $G$, i.e. $\tau \subset \tau'$, then naturally $K_{\Delta,\tau'}*\tau'$ is a subcomplex of $K_{\Delta,\tau}*\tau$. Thus there are natural pullback and pushforward maps on the Stanley-Reisner rings:
\[
(i_{G',G})^*: \widetilde{\SR}(G) \surj \widetilde{\SR}(G'), \ \ x_i \mapsto x_i, \ \ \ \ \ \ 
(i_{G',G})_*:\widetilde{\SR}(G') \to \widetilde{\SR}(G),\ \ 1\mapsto x^{\tau'\backslash \tau}.
\]
The pushforward is determined by the image of $1$ since the pullback map is a surjective ring homomorphism and the pushforward is a homomorphism as $\widetilde{\SR}(G)$-module where the module structure on $\widetilde{\SR}(G')$ is induced by the pullback map.

%%%%%%%%%%%%%%%%%%%%%%%%%%%%%%%%%%
%%%%%%%%%%%%%%%%%%%%%%%%%%%%%%%%%%
%%%%%%%%%%%%%%%%%%%%%%%%%%%%%%%%%%
\subsection{Equivariant Chen-Ruan orbifold cohomology}
We use the notation from Section \ref{comb}. Let $\mu: M \to \frakr^*$ be the moment map for the toric orbifold $[M/\sfS ]$ so that $\mu(M)=\Delta$. Recall from \cite{LT} that $^t\!B\cdot (\ ) + \eta$ embeds $\Delta$ into $\ft^*$, and $M$ is defined as the preimage of $\Delta':=\, ^t\!B(\Delta)+\eta$ under the standard moment map $\overline \mu:\C^m \to \ft^*$. The moment map $\mu$ is the composition of $\overline\mu$ with the inverse of $^t\!B\cdot (\ ) + \eta$ restricted to $\Delta':=\ ^t\!B\cdot (\Delta) + \eta$. The faces of $\Delta'$ are given by the intersections of $\Delta'$ and the coordinate planes in $\ft^*\cong \R^n$.  
\begin{lemma}
Let $G:=H_{j_1}\cap \cdots \cap H_{j_r}$ be an $(n-r)$-dimensional face. 
The global stabilizer $\sfT_{\mu^{-1}(G)}$ of $\mu^{-1}(G)$ in $\sfT$ is 
\[
\sfT_{\mu^{-1}(G)}=\{(t_1,\cdots, t_m) \in \sfT \ |\ t_i=1,\forall i \in [m]\backslash\{j_1,\cdots j_r\}\} = \bigcap_{v: \ vertex\ of\ G } \sfT_{F_v}.
\]
Furthermore $H_{\sfT}(\mu^{-1}(G)) \cong \widetilde{\SR}(G)$.
\end{lemma}
\proof 
By definition, we have $\mu^{-1}(G)=\{z \in \C^m \ |\ |z_i|^2=(^t\!B\cdot v +\eta)_i, 1\leq i \leq m, \ v \in G\}.$ Thus $z \in \mu^{-1}(G^{\circ})$ if and only if $z_i=0$ for $i \in \{j_1,\cdots, j_r\}$,  where $G^{\circ}$ is the relative interior of $G$. The second claim follows from the similar calculation as in Section \ref{comb}. \qed
\

Let $g:=(g_1,\cdots,g_m) \in \sfS  \subset \sfT$. By the lemma above, $M^g=\mu^{-1}(G_g)$ where $G_g$ is the union of faces $G$ such that $\sfT_{\mu^{-1}(G)}$ contains $g$. Let $\overline a_g:=\{ i \ |\ g_i=1\}, \overline b_g:=\{i \ |\ g_i\not=1\} \subset \{1,\cdots,m\}$. Then $g \in \sfT_{\mu^{-1}(G)}$ if and only if $\overline b_g \subset \{j_1,\cdots,j_r\}$. Therefore we have
\begin{lemma}
$G_g= \bigcap_{i \in \overline b_g} H_i$. In particular, $H_{\sfT}^*(M^g)= \widetilde{\SR}(G_g)$.
\end{lemma}
%%%%%%%%%%%%%%%%%%%%%%%%%%%%%%%%%%%
% TOMO's VERSION 5 CHANGE starts here
%%%%%%%%%%%%%%%%%%%%%%%%%%%%%%%%%%%
Now the intersection of $M^g$ and $M^h$ is given by $M^g \cap M^h =\overline\mu^{-1}(G_g \cap G_h)$ where $G_g \cap G_h = \bigcap_{ i \in \overline b_g \cup \overline b_h} H_i$. Thus the normal bundle of $M^{g,h}$ in $M$ is given by $\bigoplus_{ i \in \bar b_g \cup \bar b_h} \C\cdot \frac{\partial}{\partial z_i}$ Let $\lambda_i \in \Hom(\sfT,S^1)$ such that $\lambda_i(t)=t_i$ and define $0\leq \tilde{\lambda}_i(g) < 1$ by  $\lambda_i(g)=e^{2\pi i \cdot \tilde{\lambda}_i(g)}$. Thus the obstruction bundle and the virtual class are 
\[
\calR(g,h) = \bigoplus_{\tilde{\lambda}_i(g)+\tilde{\lambda}_i(h)+\tilde{\lambda}_i((gh)^{-1})=2, \atop{ i \in  \overline b_g \cup  \overline b_h }}  \C\cdot\frac{\partial}{\partial z_i}, \ \  \mbox{ and } \ \ 
c_{\sfT}(g,h)= \prod_{\tilde{\lambda}_i(g)+\tilde{\lambda}_i(h)+\tilde{\lambda}_i((gh)^{-1})=2, \atop{ i \in  \overline b_g \cup  \overline b_h }} x_i.
\]
The normal bundle of $M^{g,h}$ in $M^{gh}$ is given by $\bigoplus_{ i \in (\bar b_g \cup \bar b_h)\backslash \bar b_{gh}} \C\cdot \frac{\partial}{\partial z_i}$ and its Euler class is $\prod_{i \in (\bar b_g \cup \bar b_h)\backslash \bar b_{gh}} x_i$. The equivariant Chen-Ruan cohomology space is $H_{\sfR, orb}([M/\sfS ])= \bigoplus_{g \in \sfS } \widetilde{\SR}(G_g)$ where $\widetilde{\SR}(G_g):=0$ if $G_g=\phi$. Since the pullback and pushforward maps of the equivariant cohomology agree with the ones on the Stanley-Reisner rings, we find that the product is given by
\[
1_g\odot 1_h =\underbrace{ \left(\prod_{\tilde{\lambda}_i(g)+\tilde{\lambda}_i(h)+\tilde{\lambda}_i((gh)^{-1})=2, \atop{  i \in  \overline b_g \cup  \overline b_h  }} x_i \right)}_{virtual\ class}\cdot \underbrace{\left( \prod_{i\in (\overline b_g \cup \overline b_h) \backslash \overline b_{gh} }x_i\right)}_{Euler\ class\ of\atop{ nomal\ bundle}} \cdot 1_{gh},
\]
where $1_g, 1_h$ and $1_{gh}$ are the identities in the corresponding Stanley-Reisner ring. The product $\cdot$ on the right-hand side can be defined as the product in $\Z[x_1,\cdots, x_m]$, and each sector is generated by the identity element as a $\Z[x_1,\cdots, x_m]$-algebra. Thus the above formula is enough to compute the general product.
%%%%%%%%%%%%%%%%%%%%%%%%%%%%%%%%%%
%%%%%%%%%%%%%%%%%%%%%%%%%%%%%%%%%%
\subsection{Demonstration of computations} In this section, we present two computations, namely the weighted projective spaces in dimension 2, with weights $(1,1,2)$ and $(1,2,4)$.
%%%%%%%%%%%%%%%%%%%%%%%%%%%%%%%%%%
\subsubsection{The weighted projective space $\bbP^2_{(1,1,2)}$}
%%%%%%%%%%%%%%%%%%%%%%%%%%%%%%%%%%
Consider the following polytope with facets $H_i$, facet labels all (implicitly) $1$, and the corresponding primitive inward-pointing normal vectors $\rho_i$ to the facets.{\tiny
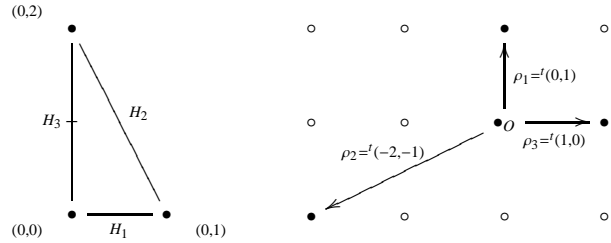
\[
\xymatrix{ 
\ar@{}[r]^{(0,2)}&\bullet\ar@{-}[dd]_{H_3}\ar@{-}[rdd]^{H_2}&&\\
&-&&\\
\ar@{}[r]_{(0,0)}&\bullet\ar@{-}[r]_{H_1}&\bullet \ar@{}[r]_{(0,1)}&
}
\ \ \ \ \ \ 
\xymatrix{
\circ&\circ&\bullet&\circ\\
\circ&\circ&\bullet_O\ar[u]_{\rho_1=^t(0,1) }\ar[r]_{\rho_3=^t(1,0) }\ar[dll]_{\rho_2=^t(-2,-1)}&\bullet\\
\bullet&\circ&\circ&\circ
}
\]
}

The polytope is given by $$\Delta=\{v \in \R^2 \ |\ \lan \rho_i, v \ran \geq -\eta_i\ ,\ i=1,2,3 \}$$ where $(\eta_1,\eta_2,\eta_3)=(0,2,0)$. The corresponding matrix $B$ is {\small $\left(\begin{array}{ccc} 0&-2&1 \\ 1&-1&0 \end{array}\right)$} and $A$ is {\small $\left(\begin{array}{c} 1\\ 1\\2 \end{array}\right)$}. Thus $M$ is given by $|z_1|^2 + |z_2|^2 + 2|z_3|^2=2$ in $\C^3$ and $\sfS = \{(t,t,t^2) \ |\ t \in \U(1)\} \subset \sfT = \U(1)^3$. The only elements $g$ of $\sfS$ such that $G_g$ is not empty are
\[
(1,1,1) \ \ \ \ \ (-1,-1,1)
\]
and the corresponding $\bar b_g, G_g, K_{G_g},  K_{G_g}*\bar b_g, \tilde{\lambda}_i$ and ages are given in the following table.
\begin{equation}
\begin{array}{c||c|c|c|c|c|c|c|c|}
g& \bar b_g & G_g & K_{G_g} & K_{G_g}*\bar b_g & \tilde{\lambda}_1 & \tilde{\lambda}_2 & \tilde{\lambda}_3 & 2\mbox{age}\\ \hline \hline
\mathbf{1}:=(1,1,1) & \phi & \Delta & K_{\Delta} & K_{\Delta} &0&0&0& 0 \\ \hline
\sigma:=(-1,-1,1) & \{1,2\} & \bullet  & \phi  & \bullet\!\!-\!\!\bullet &1/2&1/2&0& 2 \\ \hline
\end{array}
\end{equation}
Thus 
\[
H_{CR,\sfR}([M/\sfS])  = \underbrace{\frac{\Z[x_1,x_2,x_3]}{\lan x_1x_2x_3 \ran}}_{\mathbf{1}} \oplus \underbrace{\frac{\Z[x_1,x_2,x_3]}{\lan x_3\ran}}_{\sigma}.
\]
The following is the table of the multiplications between $1_{\mathbf{1}}$ and $1_{\sigma}$.
\begin{equation}
\begin{array}{c||c|c|}
 & 1_{\mathbf{1}} & 1_{\sigma}\\ \hline\hline
1_{\mathbf{1}} & 1_{\mathbf{1}}& 1_{\sigma}\\ \hline
1_{\sigma} & 1_{\sigma}& (1)\cdot(x_1x_2)\cdot 1_{\mathbf{1}}\\ \hline
\end{array}
\end{equation}
%%%%%%%%%%%%%%%%%%%%%%%%%%%%%%%%%%
%%%%%%%%%%%%%%%%%%%%%%%%%%%%%%%%%%
%%%%%%%%%%%%%%%%%%%%%%%%%%%%%%%%%%
%%%%%%%%%%%%%%%%%%%%%%%%%%%%%%%%%%
\subsubsection{The weighted projective space $\bbP^2_{(1,2,4)}$} 
%%%%%%%%%%%%%%%%%%%
Consider the following polytope with facets $(H_1,H_2,H_3)$ and with the labels $(1,1,2)$ respectively, and the corresponding primitive inward-pointing normal vectors $\rho_i$ of facets.
{\tiny
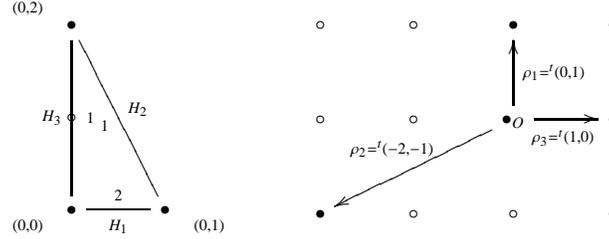
\[
\xymatrix{ 
\ar@{}[r]^{(0,2)}&\bullet\ar@{-}[dd]_{H_3}^{\ 1}\ar@{-}[rdd]^{H_2}_1&&\\
&\circ&&\\
\ar@{}[r]_{(0,0)}&\bullet\ar@{-}[r]_{H_1}^2&\bullet \ar@{}[r]_{(0,1)}&
}
\ \ \ \ \ \ \ \ 
\xymatrix{
\circ&\circ&\bullet&\circ\\
\circ&\circ&\bullet_O\ar[u]_{\rho_1=^t(0,1) }\ar[r]_{\rho_3=^t(1,0)}\ar[dll]_{\rho_2=^t(-2,-1)}&\bullet\\
\bullet&\circ&\circ&\circ
}
\]
}
The corresponding matrix $B$ is ${\small \left(\begin{array}{ccc} 0&-2&1 \\ 2&-1&0 \end{array}\right)}$ and $A$ is ${\small \left(\begin{array}{c} 1\\ 2\\4 \end{array}\right)}$. The hyperplanes defining $\Delta$ is still given by $(\eta_1,\eta_2,\eta)=(0,2,0)$. Thus $\sfS=\{(t,t^2,t^4)\}$ and $M$ is given by the equation $|z_1|^2 + 2|z_2|^2 + 4|z_3|^2=4$. The corresponding $\bar b_g, G_g, K_{G_g},  K_{G_g}*\bar b_g, \tilde{\lambda}_i$ and ages are computed in the following table.
\begin{equation}
\begin{array}{c||c|c|c|c|c|c|c|c|}
g& \bar b_g & G_g & K_{G_g} & K_{G_g}*\bar b_g & \tilde{\lambda}_1 & \tilde{\lambda}_2 & \tilde{\lambda}_3&2\mbox{age}\\ \hline \hline
\mathbf{1}=(1,1,1) & \phi & \Delta & K_{\Delta} & K_{\Delta}										&0&0&0& 0\\ \hline
\xi=(\sqrt{1},-1,1) & \{1,2\} & \bullet  & \phi  & \bullet\!\!-\!\!\bullet 							&1/4&1/2&0&3/2\\ \hline
\xi^2=(-1,1,1) &         \{1\}   & \bullet\!\!-\!\!\bullet & \bullet\  \bullet & \bullet\!\!-\!\!\bullet\!\!-\!\!\bullet	&1/2&0&0&1\\ \hline
\xi^3=(-\sqrt{1},-1,1)&\{1,2\}& \bullet  & \phi  & \bullet\!\!-\!\!\bullet 							&3/4&1/2&0&5/2\\ \hline
\end{array}
\end{equation}
Thus 
\[
H_{CR,\sfR}([M/\sfS]) = \underbrace{\frac{\Z[x_1,x_2,x_3]}{\lan x_1x_2x_3\ran }}_{\mathbf{1}} \oplus \underbrace{\frac{\Z[x_1,x_2,x_3]}{\lan x_3\ran }}_{\xi}  \oplus \underbrace{\frac{\Z[x_1,x_2,x_3]}{\lan x_2x_3\ran}}_{\xi^2} \oplus \underbrace{\frac{\Z[x_1,x_2,x_3]}{\lan x_3\ran }}_{\xi^3}.
\]
The following is the table of the multiplications of $1_{\mathbf{1}}, 1_{\xi}, 1_{\xi^2},1_{\xi^3}$:
\[
\begin{array}{c||c|c|c|c|}
g\backslash h & 1_{\mathbf{1}} &1_{\xi}&1_{\xi^2} & 1_{\xi^3}   \\ \hline\hline
1_{\mathbf{1}} 		&1_{\mathbf{1}}&1_{\xi} &1_{\xi^2} &1_{\xi^3} \\ \hline
1_{\xi}	&1_{\xi^2} & (1)\cdot (x_2)\cdot 1_{\xi^2} & 1_{\xi^3}  &(1)\cdot(x_1x_2)\cdot 1_{\mathbf{1}}  \\ \hline
1_{\xi^2}&1_{\xi^2} & 1_{\xi^3}& (1)\cdot(x_1)\cdot 1_{\mathbf{1}} &(x_1)\cdot(1)\cdot 1_{\xi}   \\ \hline
1_{\xi^3} 	&1_{\xi^3} & (1)\cdot(x_1x_2)\cdot 1_{\mathbf{1}}&(x_1)\cdot(1) \cdot 1_{\xi}  &(x_1)\cdot (x_2)\cdot 1_{\xi^2}  \\ \hline
\end{array}
\]
%%%%%%%%%%%%%%%%%%%%%%%%%%%%%%%
\subsection{Presentations of $H_{\sfR, orb}([M/\sfS ])$ as a subring}
%%%%%%%%%%%%%%%%%%%%%%%%%%%%%%%

From our main result, $H^*_{\sfR, orb}([M/\sfS ])$ is a subring of
\[
\calNH^*_{\sfR}(\nu[F/\sfS ])= \bigoplus_{g \in \sfS }  \bigoplus_{v \in G_g} \Z[x_i,  i \in \bv] =  \bigoplus_{g \in \sfS } \left\{\left(p_v\right)_{\overline b_g \subset \bv}, p_v \in \Z[x_i,  i \in \bv] \right\}
\]
where $\bv$ corresponds to $v$ by $v=\bigcap_{i\in\bv} H_i$. Note that $v\in G_g \Leftrightarrow \overline b_g \subset \bv$. The product $\left(p_v\right)_{\overline b_g \subset \bv}\star \left(p_w\right)_{\overline b_h \subset \bw} \in \bigoplus_{\overline b_{gh} \subset \bu} \Z[x_i,  i \in \bu]$ can be computed  by its $u$-component 
\begin{eqnarray*}
&&\left.\left(p_v\right)_{\overline b_g \subset \bv}\star \left(p_w\right)_{\overline b_h \subset \bw}\right|_u \\
&=& p_u\cdot q_u\cdot \left(\prod_{\tilde{\lambda}_i(g)+\tilde{\lambda}_i(h)+\tilde{\lambda}_i((gh)^{-1})=2, \atop{  i \in  \overline b_g \cup  \overline b_h  }} x_i\right)\left(\prod_{i \in  (\overline b_g \cup \overline b_h) \backslash  \overline b_{gh}} x_i\right) 
\end{eqnarray*}
if $\overline b_g \cup \overline b_h \subset \bu$ and otherwise is zero.

\end{document}